\documentclass[11pt,a4paper,amssymb,amsmath, tightenlines]{article}

\usepackage{amsmath, amssymb, amsthm, latexsym, mathrsfs, url, epsfig, graphics, float, setspace}
\usepackage[usenames, dvipsnames, svgnames, x11names]{xcolor}
\usepackage{bbm}
\usepackage{graphicx}
\usepackage{multirow}
\usepackage{subfigure} 
\usepackage{sectsty}
\usepackage[T1]{fontenc}
\usepackage[bitstream-charter]{mathdesign}
\usepackage{pgfplots}
\usepackage{bm}
\usepackage[round,comma,authoryear,sort&compress]{natbib} 
\allowdisplaybreaks[1]
\newcommand{\Hquad}{\hspace{0.5em}} 

\newtheorem{definition}{Definition}[section]
\newtheorem{remark}{Remark}[section]
\newtheorem{proposition}{Proposition}[section]
\newtheorem{lemma}{Lemma}[section]
\newtheorem{corollary}{Corollary}[section]
\newtheorem{example}{Example}[section]
\numberwithin{equation}{section}
\newtheorem{theorem}{Theorem}[section]

\newcommand{\e}{{\rm e}}

\newcommand{\F}{\mathscr{F}}

\newcommand{\R}{\mathbb{R}}

\newcommand{\rd}{\textup{d}}

\newcommand{\indi}[1]{1\hspace{-.09cm}\textup{\textrm{l}}}

\usepackage{algorithm}
\usepackage[noend]{algpseudocode}
\makeatletter
\def\BState{\State\hskip-\ALG@thistlm}
\makeatother

\begin{document}
\title{\vspace{-1cm}\bf Multiple barrier-crossings of an \\ Ornstein-Uhlenbeck diffusion in consecutive periods}
\author{Yupeng Jiang$^{\dag}$, Andrea Macrina{$^{\dag\, \ddag}$\footnote{Corresponding author: a.macrina@ucl.ac.uk}\ }, Gareth W. Peters{$^{\S}$} \\ \\ {$^{\dag}$Department of Mathematics, University College London} \\ {London WC1E 6BT, United Kingdom} \\ \\ {$^{\ddag}$ African Institute for Financial Markets and Risk Management} \\ {University of Cape Town} \\ {Rondebosch 7701, South Africa} \\ \\ {$^{\S}$}Department of Actuarial Mathematics and Statistics \\Heriot-Watt University\\ Edinburgh EH14 4AS, United Kingdom}
\date{15 October 2020}

\maketitle
\vspace{-0.75cm}
\begin{abstract}
\noindent  
\\\vspace{-1.25cm}\\

We investigate the joint distribution and the multivariate survival functions for the maxima of an Ornstein-Uhlenbeck (OU) process in consecutive time-intervals. A PDE method, alongside an eigenfunction expansion, is adopted with which we first calculate the distribution and the survival functions for the maximum of a homogeneous OU-process in a single interval. By a deterministic time-change and a parameter translation, this result can be extended to an inhomogeneous OU-process. Next, we derive a general formula for the joint distribution and the survival functions for the maxima of a continuous Markov process in consecutive periods. With these results, one can obtain semi-analytical expressions for the joint distribution and the multivariate survival functions for the maxima of an OU-process, with piecewise constant parameter functions, in consecutive time periods. The joint distribution and the survival functions can be evaluated numerically by an iterated quadrature scheme, which can be implemented efficiently by matrix multiplications. Moreover, we show that the computation can be further simplified to the product of single quadratures by imposing a mild condition. Such results may be used for the modelling of heatwaves and related risk management challenges.
\\\vspace{-0.2cm}
\end{abstract}

\section{Introduction}
The Ornstein-Uhlenbeck (OU) process is a well-known diffusion process, widely used in physics, finance, biology and other fields. Due to its extensive use, the study of its first-passage-time (FPT) arises, naturally. The FPT density of a homogeneous OU-process to particular cases of barrier functions can be found in closed-form. For example, if the barrier is equal to the OU long-term mean, its closed-form probability density function (PDF) can be found in \cite{GoingJaeschke}, \cite{Ricciardi} and in \cite{yi}. However, it is more involved to obtain the PDF of the FPT of the homogeneous OU-process to an arbitrary constant barrier. In \cite{Leblanc} it is claimed that the closed-form solution was found, but in \cite{GoingJaeschke} it is pointed out that the results in \cite{Leblanc} are wrong due to the errors encountered when using the property of 3D Bessel bridges, see \cite{PitmanYor1981} and \cite{PitmanYor1982}. In \cite{GoingJaeschke} a Bessel bridge representation for the FPT-PDF of a homogeneous OU-process to an arbitrary constant barrier is provided. The analytical expression of the moment generating function for the homogeneous OU-FPT has been well-studied, see for instance \cite{Alili}, \cite{Patie} and \cite{Ricciardi}, and one may obtain an infinite-series representation for the PDF of the FPT of a homogeneous OU-process crossing an arbitrary constant barrier by the inverse Laplace transform, see \cite{Alili}. The same infinite-series representation is also obtained in \cite{Linetsky2004d} based on the spectral theory for options pricing in \cite{Linetsky2004b}. However, to our knowledge, the properties of the infinite-series solution, especially the tail behaviour, have not yet been studied. The tail behaviour of the infinite-series representation is practically important since it determines whether one can use the truncated series as a robust approximation.
\

The case of a time-inhomogeneous OU-process passing a time-dependent barrier tends to be more complicated. In \cite{Tuckwell} the FPT of a time-homogeneous Ito process, where the barrier function must satisfy a first-order linear ODE, is studied. Under such conditions, a solution is provided using numerical PDE methods. Work by \cite{Durbin} focusses on deriving an explicit expression for the density of the FPT of a continuous Gaussian process to a general boundary under mild conditions. The first-passage density of a Brownian motion to a curved boundary is given by an integral series in \cite{DurbinWilliams} and a numerical method is produced to compute the first-passage probability with high accuracy. In \cite{Buonocore}, it is shown that the FPT-PDF of a diffusion process passing a time-dependent boundary satisfies a Volterra integral equation of the second kind involving two arbitrary continuous functions. By this method, the FPT-PDF for a homogeneous OU-process passing some special barrier specifications, e.g. the barrier function is hyperbolic with respect to time, can be found analytically. Methods to obtain numerical solutions to the first-passage problem involving Gauss-Markov processes are developed in \cite{DiNardo} and \cite{Giorno}. The smoothness of the FPT distribution is investigated in \cite{Lehmann} and an integral equation is provided for the FPT density function of a continuous Markov process. In \cite{Gutierrez}, the integral equation approach is generalized to time-inhomogeneous diffusion processes. The FPT-PDF of a time-inhomogeneous diffusion process passing a constant barrier can be obtained numerically by solving a PDE, see e.g. \cite{2ndCourse} and \cite{Eigen}. In \cite{lo} the Fokker-Planck equation associated with an inhomogeneous OU-process passing a time-dependent barrier is studied and the method of images to derive the solution is introduced. However, the generalization to an unconstrained time-dependent barrier cannot be produced due to the strict conditions imposed by the method of images. In \cite{HernandezdelValle} one finds the FPT of Ito processes whose local drift can be modelled in terms of a solution of the Burgers equation. However, the OU-process class does not belong to such a process family. In \cite{LipKau} a semi-analytical method is developed to calculate the first hitting-time of an OU-process by use of heat potentials, and in \cite{Martin} the behaviour of the first-passage time of a mean-reverting process over short and long time periods is described by an approximation formula.
\

 {\bf Motivation}. Since all continuous functions can be approximated to arbitrary precision by piece-wise constant functions, it is worthwhile to study the FPT of a homogeneous or inhomogeneous OU-process passing a piece-wise constant barrier function. In this paper, one of the main focuses is put on the joint probability that the running maximum is above arbitrary fixed thresholds in pre-specified consecutive time intervals. Such a probabilistic problem arises for example in applications to environmental and climate risk, to which the insurance industry, but more importantly general global welfare, is exposed. Heat waves, or repeated prolonged periods of droughts, can have substantial impact on economies, be these regional or (supra-)national. A heat wave is an event that often is defined by the temperature passing a pre-specified threshold on a number of consecutive days. This is an unequivocal case where the joint probability of the running maximum of a stochastic process passing a fixed arbitrary barrier in consecutive intervals is necessary to address an important real-world challenge. However, to our knowledge, the mathematical problem has neither been formulated, nor tackled or solved before and the needed mathematical theory has not been developed, either. 
\

{\bf Main results}. In this work, we study the multivariate survival function associated with an OU-process crossing arbitrary barriers in multiple time intervals. In Section \ref{homo}, we adopt a PDE approach to deduce the infinite series representation of the survival function for the FPT of a homogeneous OU-process with lower reflection barrier passing a constant upper barrier. By considering the lower reflection barrier set at $-\infty$, we produce the same infinite series representation as in \cite{Alili} and \cite{Linetsky2004d}. This can be viewed as a generalization and an alternative derivation of the infinite series representation. Moreover, we analyze the distributional properties of the deduced survival function, especially its tail behaviour and the truncation error. In Section \ref{inhomo}, we provide a theorem that transforms the FPT of an inhomogeneous OU-process passing a time-dependent barrier to the FPT of a homogeneous OU-process with a different time-dependent barrier. This transfers the time-inhomogeneity from the process to the time-dependent barrier, which simplifies the original problem. In Section \ref{multi} we deduce an integral representation of the joint distribution and joint survival function for the maxima of a continuous Markov process in consecutive intervals.  Although the work in this paper is based on the OU-process, the integral representation derived in this section opens up avenues towards the derivation of the probability of multiple crossings in consecutive periods of generic (mean-reverting) continuous Markov processes, thus generalizing the results derived in the present paper. With the knowledge of the integral representation, alongside the FPT density function and the numerical integration method, the joint distribution and joint survival function for the maxima of an OU-process with piece-wise constant parameters in consecutive intervals can be efficiently obtained. We also show that under certain assumptions, the nested integration can be further simplified to become a product of single integrals, which leads to improved computational efficiency. Finally, in Section \ref{num} we present the quadrature scheme and the Monte Carlo integration method for the numerical integration. Comparing with the direct Monte Carlo approach, the results obtained by either the quadrature scheme or the Monte Carlo integration method show higher accuracy and robustness. This is especially true in the {\it rare-event} cases, where the direct Monte Carlo approach fails to reduce the approximation error, efficiently. 

\section{Survival function for the FPT of a homogeneous OU-process passing a constant barrier}\label{homo}
We begin by considering the first-passage-time (FPT) of a homogeneous Ornstein-Uhlenbeck (OU) process crossing a constant barrier. In order to deduce the analytical FPT survival function of a homogeneous OU-process, we adopted a PDE approach. By considering an alternative derivation, we generalise the result in \cite{Alili} and \cite{Linetsky2004d} in that we derive the infinite series representation for the survival function by setting the lower reflection barrier at negative infinity. Moreover, the distributional properties of the deduced survival function are analysed, especially its tail behaviour and the truncation error.
\begin{definition}\label{homoOU}
	An $\mathbb{R}$-valued stochastic process $\left( X_t \right)_{t\geq 0}$ is called a homogeneous OU-process if it satisfies the stochastic differential equation
	\begin{align}\label{homoOUeqn}
	\rd X_t = \left( \mu - \lambda X_t \right) \rd t + {\sigma}\rd W_t,
	\end{align}
	where $	X_0 = x \in \mathbb{R}$, for $\mu\in \mathbb{R}$, $\lambda>0$ and $\sigma >0 $, where $(W_t)_{t\geq 0}$ is a Brownian motion on the probability space $\left( \Omega, \mathscr{F}, \mathbb{P} \right)$. When $\mu = 0$, $\lambda = \sigma = 1$, we call the process $(\widetilde{X}_t)_{t\geq 0}$ that satisfies
	\begin{align}\label{simOU}
	\rd \widetilde{X}_t =  -\widetilde{X}_t \rd t + \rd W_t,
	\end{align}
	a standardised OU-process.
\end{definition}

\begin{definition}\label{FPT}
	The first-passage-time (FPT) of a continuous process $\left(X_t\right)_{t\geq 0}$ to an upper constant barrier $b>X_0=x$ is defined by
	$\tau_{X,b}:=\inf\left\{ t\geq 0 : X_t \geq b \right\}$. The survival function of $\tau_{X, b}$, denoted by $\bar{F}_{\tau_{X,b}}(t;x)$, is given by
	$\bar{F}_{\tau_{X,b}}(t;x) = \mathbb{P}\left( \tau_{X, b} \geq t\, | \, \, X_0 = x \right)$. 
\end{definition}

As stated in \cite{Alili}, \cite{Patie} and \cite{Linetsky2004d}, if $\left(X_t\right)_{t\geq 0}$ is a homogeneous OU-process, the random variable $\tau_{X,b}$ is ``properly'' defined in the sense that
$\mathbb{P}\left( \tau_{X,b} < \infty \right) = 1$. 

Next we present a relation between the survival functions of the FPTs for two different homogeneous OU-processes. With Lemma \ref{OUtransform}, if one knows the FPT distribution of a homogeneous OU-process to a given barrier, the FPT distribution of another homogeneous OU-process to a shifted barrier can also be obtained. 

\begin{lemma}\label{OUtransform}
	The random variable $\tau_{X, b}$ is equal to $\tau_{\widetilde{X}, \tilde{b}}$ in distribution for
	$$\tilde{t} = \lambda t, \quad \tilde{x} = \sqrt{\frac{\lambda}{\sigma^2}}\left( x-\frac{\mu}{\lambda} \right), \quad \tilde{b} = \sqrt{\frac{\lambda}{\sigma^2}}\left( b-\frac{\mu}{\lambda} \right),$$  
	that is, 
	$
	\bar{F}_{\tau_{X, b}}\left(t, x\right) = \bar{F}_{\tau_{\widetilde{X}, \tilde{b}}}\left(\tilde{t},\tilde{x}\right).
	$
\end{lemma}

\begin{proof}
	We have
	\begin{align*}
	\bar{F}_{\tau_{X, b}}\left(t, x\right) =& \mathbb{P}\left( \sup_{s\in[0,t]} X_s < b \Big\vert\, X_0 = x \right).
	\end{align*}
	Then by a change of time, it follows
	\begin{align*}
	\mathbb{P}\left( \sup_{s\in[0,t]} X_s < b \Big\vert\, X_0 = x \right) = & \mathbb{P}\left( \sup_{s\in[0,\lambda t]} X_{{s}/{\lambda}} < b \Big\vert\, X_0 = x \right).
	\end{align*}
Furthermore,
	\begin{align*}
	\mathbb{P}\left( \sup_{s\in[0,\lambda t]} X_{{s}/{\lambda}} < b \Big\vert\, X_0 = x \right)&= \mathbb{P}\left( \sup_{s\in[0,\lambda t]} \sqrt{\frac{\lambda}{\sigma^2}} X_{{s}/{\lambda}} < \sqrt{\frac{\lambda}{\sigma^2}} b\, \Big\vert\, \sqrt{\frac{\lambda}{\sigma^2}} X_0 = \sqrt{\frac{\lambda}{\sigma^2}} x \right) \\
	&= \mathbb{P}\left( \sup_{s\in[0,\lambda t]} \left( \sqrt{\frac{\lambda}{\sigma^2}} X_{{s}/{\lambda}} - \frac{\mu}{\sigma\sqrt{\lambda}} \right) < \tilde{b} \,\Big\vert\, \sqrt{\frac{\lambda}{\sigma^2}} X_0 -\frac{\mu}{\sigma\sqrt{\lambda}} = \tilde{x} \right).
	\end{align*}
	The dynamics of the process $(\sqrt{\lambda/\sigma^2}\, X_{{s}/{\lambda}} -\mu/(\sigma\sqrt{\lambda}) )_{s\geq 0}$ are given by
	\begin{align*}
	\rd \left(  \sqrt{\frac{\lambda}{\sigma^2}} X_{{s}/{\lambda}} - \frac{\mu}{\sigma\sqrt{\lambda}} \right) =& \sqrt{\frac{\lambda}{\sigma^2}} \rd X_{{t}/{\lambda}} 
	= - \left( \sqrt{\frac{\lambda}{\sigma^2}} X_{{t}/{\lambda}} - \frac{\mu}{\sigma\sqrt{\lambda}} \right) \rd t + \rd W_t.
	\end{align*}
	This means that, in law, the process $(\sqrt{\lambda/\sigma^2}\, X_{{s}/{\lambda}} -\mu/(\sigma\sqrt{\lambda}) )_{s\geq 0}$ is a standardised OU-process. Therefore,
	\begin{align*}
	\mathbb{P}\left( \sup_{s\in[0,t]} X_s < b \Big\vert\, X_0 = x \right)  = \mathbb{P}\left( \sup_{s\in[0,\tilde{t}]} \widetilde{X}_s < \tilde{b} \Big\vert\, \widetilde{X}_0 = \tilde{x} \right), 
	\end{align*}
	that is,
	$\bar{F}_{\tau_{X, b}}\left(t, x\right) = \bar{F}_{\tau_{\widetilde{X}, \tilde{b}}}\left(\tilde{t},\tilde{x}\right)$.
\end{proof}

Lemma \ref{OUtransform} provides a relationship between the survival functions---that is between the distributions of the FPTs---of the homogeneous, respectively, the standardized OU-process. In order to calculate the FPT survival function for a homogeneous OU-process, one can first calculate the FPT survival function for a standardized OU-process. Therefore, from now on in this section, we consider the case of a standardized OU-process. 

\subsection{The FPT survival function of the standardized OU-process to a constant barrier}

The FPT survival function of the standardized OU-process to a constant upper barrier can be characterized by the following PDE problem. On the space $C^{1,2}\left( [0, \infty), (-\infty, \tilde b] \right)$, the function $\bar{F}_{\tau_{\widetilde{X}, \tilde{b}}}\left(\tilde{t}, \tilde{x}\right)$ satisfies the PDE
\begin{align}\label{specPDE}
\frac{\partial \bar{F}_{\tau_{\widetilde{X}, \tilde{b}}}}{\partial \tilde t} = \mathscr{A} \bar{F}_{\tau_{\widetilde{X}, \tilde{b}}}
\end{align}
subject to the initial and boundary conditions 
\begin{align}
\bar{F}_{\tau_{\widetilde{X}, \tilde{b}}}\left(0, \tilde{x}\right) &= 1, \label{specIni}\\
\bar{F}_{\tau_{\widetilde{X}, \tilde{b}}}\left(\tilde t, \tilde{b}\right) &= 0. \label{specBnd1}
\end{align}
Here $\mathscr{A}$ is the infinitesimal operator of a standardized OU-process $(\widetilde{X}_{\tilde t})_{\tilde t\geq 0}$ given by
\begin{align*}
\mathscr{A} = - \tilde x \frac{\partial }{\partial \tilde x} + \frac{1}{2} \frac{\partial^2 }{\partial \tilde x^2}.
\end{align*}
In order to solve this PDE, we add the lower boundary condition
\begin{align}\label{specBnd2}
\frac{\partial \bar{F}_{\tau_{\widetilde{X}, \tilde{b}}}}{\partial \tilde x}\left(\tilde t, \tilde a\right) &= 0,
\end{align}
which is the condition for a reflecting lower boundary at location $\tilde{a}<\tilde{b}$. 

\begin{proposition}\label{reflectPDE}
	The analytic solution to the PDE (\ref{specPDE}), subject to the initial condition and boundary conditions (\ref{specIni}), (\ref{specBnd1}) and (\ref{specBnd2}) is given by
	\begin{align*}
	\bar{F}_{\tau_{\widetilde{X}, \tilde{b}}}(\tilde t, \tilde x) = \sum\limits_{k = 1}^{\infty} c_k \e^{-\alpha_k \cdot \tilde t} H(\alpha_k, \tilde x;\tilde a),
	\end{align*}
	for $k\in\mathbb{N}$, where
	\begin{align*}
	H(\alpha_k, \tilde x;\tilde a) =& 
	\frac{2^\alpha \sqrt{\pi}}{\Gamma\left( \frac{1-\alpha}{2} \right)}\left({\Hquad}_1F_{1}\left( -\frac{\alpha_k}{2};\frac{1}{2};\tilde x^2 \right) + y(\alpha_k,\tilde a) \tilde x {\Hquad}_1F_{1}\left( \frac{1-\alpha_k}{2};\frac{3}{2};\tilde x^2 \right)\right),\\
	y(\alpha_k,\tilde a) =& \frac{2 \alpha_k \tilde a {\Hquad}_1F_{1}\left( \frac{2-\alpha_k}{2};\frac{3}{2};\tilde a^2 \right) }{{\Hquad}_1F_{1}\left( \frac{1-\alpha_k}{2};\frac{3}{2};\tilde a^2 \right) + \frac{2}{3}(1-\alpha_k)\tilde a^2 {\Hquad}_1F_{1}\left( \frac{3-\alpha_k}{2};\frac{5}{2};\tilde a^2 \right)}.
	\end{align*}
	Here, ${\Hquad}_1F_{1}$ is the confluent hypergeometric function of the first kind and $\alpha_k$ are the ordered solutions to the equation
	\begin{align}
	{\Hquad}_1F_{1}\left( -\frac{\alpha}{2};\frac{1}{2};\tilde b^2 \right) + y(\alpha,\tilde a) \tilde b {\Hquad}_1F_{1}\left( \frac{1-\alpha}{2};\frac{3}{2};\tilde b^2 \right) = 0 \label{y_alpha}
	\end{align}
	with respect to $\alpha$. Furthermore, the coefficient $c_k$ is given by $$c_k =-1/[\alpha_k\, {\partial_{\alpha_k} H}( \alpha_k, \tilde x;\tilde a)].$$
\end{proposition}
\begin{remark}
	This proposition provides a generalization to the infinite series representation in \cite{Alili} and \cite{Linetsky2004d}. It recovers the previous result when $\tilde{a}\rightarrow-\infty$, which will be shown in Theorem \ref{OUhit}. We refer to \cite{Kent} for eigenvalue expansions as a technique to solve first-passage time problems in a diffusion setting.
\end{remark}

\begin{proof}
	By the method of eigenfunction expansion, $\bar{F}_{\tau_{\widetilde{X}, \tilde{b}}}(\tilde t, \tilde x)$ admits the following representation
	\begin{align*}
	\bar{F}_{\tau_{\widetilde{X}, \tilde{b}}}(\tilde t, \tilde x) = \sum\limits_{k = 1}^{\infty} c_k \e^{-\alpha_k \tilde t} \phi_k(\tilde x), 
	\end{align*}
	where $c_k$ are the constant coefficients, and $\alpha_k$ and $\phi_k(\tilde x)$ are the eigenvalues and eigenfunctions that satisfy the general eigenfunction equation
	\begin{align}\label{eigenEqn}
	\mathscr{A}\phi_k(\tilde x) = -\alpha_k \phi_k(\tilde x)
	\end{align}
	subject to $\phi_k^\prime(\tilde a) = \phi_k(\tilde b) = 0.$ The pair $\left(\phi(\cdot), \alpha\right)$ satisfies
	\begin{align}\label{hermite}
	\frac{\rd ^2 \phi}{\rd \tilde x^2} - 2\tilde x \frac{\rd \phi}{\rd \tilde x} +  2\alpha \phi = 0
	\end{align}
	subject to 
	\begin{align}\label{hermiteCond}
	\phi^\prime(\tilde a) = \phi(\tilde b) = 0.
	\end{align}
	As shown in \cite{ODEs}, the ODE (\ref{hermite}) is known as the Hermite differential equation, whose general solution is given by
	\begin{align}\label{hermiteSol}
	\phi(\tilde x) = A {\Hquad}_1F_{1}\left( -\frac{\alpha}{2};\frac{1}{2};\tilde x^2 \right) + B \tilde x {\Hquad}_1F_{1}\left( \frac{1 - \alpha}{2};\frac{3}{2};\tilde x^2 \right) 
	\end{align}
	where $A$ and $B$ are independent of $\tilde{x}$. After substituting Equation (\ref{hermiteSol}) into condition (\ref{hermiteCond}), we obtain the system
	\begin{align*}
	\left\{
	\begin{array}{l}
	A  {\Hquad}_1F_{1}\left( -\frac{\alpha}{2};\frac{1}{2};\tilde b^2 \right) + B \tilde b  {\Hquad}_1F_{1}\left( \frac{1 - \alpha}{2};\frac{3}{2};\tilde b^2 \right) = 0, \\
	\\
	B \left[ {\Hquad}_1F_{1}\left( \frac{1-\alpha}{2};\frac{3}{2};\tilde a^2 \right) + \frac{2}{3}(1-\alpha)\tilde a^2 {\Hquad}_1F_{1}\left(\frac{3-\alpha}{2};\frac{5}{2};\tilde a^2 \right) \right] =2 A \alpha \tilde a {\Hquad}_1F_{1}\left( \frac{2-\alpha}{2};\frac{3}{2};\tilde a^2 \right).
	\end{array}
	\right.
	\end{align*}
	Therefore, the eigenvalues $\alpha_k$ must be the zeros of the equation
	\begin{align*}
	{\Hquad}_1F_{1}\left( -\frac{\alpha}{2};\frac{1}{2};\tilde b^2 \right) + y(\alpha,\tilde a) \tilde b  {\Hquad}_1F_{1}\left( \frac{1-\alpha}{2};\frac{3}{2};\tilde b^2 \right) = 0
	\end{align*}
	with respect to $\alpha$. We write
	\begin{align*}
	\phi_k(\tilde x) =& H(\alpha_k, \tilde x;\tilde a)\\ =& \frac{2^\alpha \sqrt{\pi}}{\Gamma\left( \frac{1-\alpha}{2} \right)}\left({\Hquad}_1F_{1}\left( -\frac{\alpha_k}{2};\frac{1}{2};\tilde x^2 \right) + y(\alpha_k,\tilde a) \tilde x  {\Hquad}_1F_{1}\left( \frac{1-\alpha_k}{2};\frac{3}{2};\tilde x^2 \right)\right), 
	\end{align*}
	which is convenient for later use. Similar to \cite{Linetsky2004b}, the coefficient of each term can be calculated by tedious but simple steps yielding
	\begin{align*}
	c_k =-1/[\alpha_k\, {\partial_{\alpha_k} H}( \alpha_k, \tilde x;\tilde a)].
	\end{align*} 
\end{proof}
Proposition \ref{reflectPDE} gives the survival function of the FPT for a homogeneous OU-process passing a given upper barrier subject to a lower reflection boundary. We re-derive the formulae after removing the lower reflection boundary by taking a limit in the following theorem. This can be treated as a different derivation of the infinite series representation in \cite{Alili} and \cite{Linetsky2004d} based on relaxing specific conditions. Here, the definition of the Hermite function $\mathscr{H}_\alpha(x)$ is given in \cite{Abramowitz}.

\begin{theorem}\label{OUhit}
	The analytic solution to the PDE (\ref{specPDE}), subject to the initial and boundary conditions (\ref{specIni}) and (\ref{specBnd1}), respectively, is given by
	\begin{align*}
	\bar{F}_{\tau_{\widetilde{X}, \tilde{b}}}(\tilde t, \tilde x) = \sum\limits_{k = 1}^{\infty} c_k \e^{-\alpha_k \tilde t} \mathscr{H}_{\alpha_k}\left( -\tilde x \right)
	\end{align*}
	for $k \in \mathbb{N}$. Here, $\mathscr{H}_\alpha(\cdot)$ is the Hermite function with parameter $\alpha$, and $\alpha_k$ are the solutions to the equation
	$\mathscr{H}_{\alpha}\left( -\tilde b \right) = 0$
	with respect to $\alpha$, and 
	$c_k = -1/[\alpha_k \cdot {\partial_{\alpha_k} \mathscr{H}_{\alpha_k}}( -\tilde b)]$. 
\end{theorem}

\begin{proof}
	By \cite{Abramowitz}, we have the asymptotic
	${\Hquad}_1F_{1}\left( x;y;z \right) \sim \Gamma(y)\e^z z^{x-y}/\Gamma(x)$
	for $z\rightarrow \infty$, that is,
	\begin{align*}
	\lim\limits_{z \rightarrow \infty} \frac{{\Hquad}_1F_{1}\left( x;y;z \right)\Gamma(x)}{\Gamma(y){\e^z z^{x-y}}} = 1.
	\end{align*}
	Therefore,
	\begin{align*}
	\lim\limits_{\tilde{a} \rightarrow\, -\infty} y(\alpha, \tilde a) =& \lim\limits_{\tilde{a} \rightarrow\, -\infty}\frac{2 \alpha \tilde a {\Hquad}_1F_{1}\left( \frac{2-\alpha}{2};\frac{3}{2};\tilde a^2 \right) }{{\Hquad}_1F_{1}\left( \frac{1-\alpha}{2};\frac{3}{2};\tilde a^2 \right) + \frac{2}{3}(1-\alpha)\tilde a^2 {\Hquad}_1F_{1}\left( \frac{3-\alpha}{2};\frac{5}{2};\tilde a^2 \right)}\\
	=& \frac{ 2\alpha}{\lim\limits_{\tilde a \rightarrow\, -\infty} \frac{{\Hquad}_1F_{1}\left( \frac{1-\alpha}{2};\frac{3}{2};\tilde a^2 \right)}{{\Hquad}_1F_{1}\left( \frac{2-\alpha}{2};\frac{3}{2};\tilde a^2 \right)\tilde a} + \frac{2}{3}(1-\alpha) \lim\limits_{\tilde a \rightarrow\, -\infty} \frac{{\Hquad}_1F_{1}\left( \frac{3-\alpha}{2};\frac{5}{2};\tilde a^2 \right)\tilde a}{{\Hquad}_1F_{1}\left( \frac{2-\alpha}{2};\frac{3}{2};\tilde a^2 \right)} }
	=  \frac{2\Gamma(\frac{1-\alpha}{2})}{\Gamma(-\frac{\alpha}{2})},
	\end{align*}
	in particular, when $\alpha = \alpha_k$.
	For $\tilde{a} \rightarrow\, -\infty$ and $\alpha = \alpha_k$, the eigenvalues $\alpha_k$ are required to satisfy
	\begin{align*}
	2^{\alpha}\sqrt{\pi}\left(\frac{{\Hquad}_1F_{1}\left( -\frac{\alpha}{2};\frac{1}{2};\tilde b^2 \right)}{\Gamma\left( \frac{1-\alpha}{2} \right)} + 2 \tilde b\frac{{\Hquad}_1F_{1}\left( \frac{1-\alpha}{2};\frac{3}{2};\tilde b^2 \right)}{\Gamma\left( -\frac{\alpha}{2} \right)} \right)= 0, 
	\end{align*}
	which turn out to be the zeros of the Hemite function $\mathscr{H}_\alpha(- \tilde b)$ with respect to $\alpha$:
	\begin{align*}
	\mathscr{H}_\alpha(-\tilde b):=
	2^\alpha\sqrt{\pi}\left[\frac{{\Hquad}_1F_{1}\left( -\frac{\alpha}{2};\frac{1}{2};\tilde b^2 \right)}{\Gamma\left( \frac{1-\alpha}{2} \right)} +2\tilde b\frac{{\Hquad}_1F_{1}\left( \frac{1-\alpha}{2};\frac{3}{2};\tilde b^2 \right)}{\Gamma\left( -\frac{\alpha}{2} \right)}\right]= 0.
	\end{align*}
	Thus, the eigenfunctions are represented by $\phi_k(\tilde x) =\mathscr{H}_{\alpha_k}(-\tilde x)$.
	The coefficients can be then obtained by Proposition \ref{reflectPDE}. 
\end{proof}
\begin{remark}
	The Hermite function $\mathscr{H}_\alpha(x)$ is equal to the limit
	$\lim\limits_{\tilde a \rightarrow -\infty} H(\alpha, x; \tilde a)$
	in Proposition \ref{reflectPDE}.
\end{remark}
{
	\begin{example}\rm{
			Here we consider the PDF for the FPT of a standardized OU-process hitting the upper barrier $\tilde b$ with different lower reflection barriers $\tilde a$. Figure \ref{zero_fpt} shows that the distance between ordered eigenvalues tends to increase, regardless of the value $\tilde{a}$ takes.  We can observe from Figure \ref{pdf_fpt} that when $\tilde a$ becomes smaller, the PDF with lower reflection barrier approaches the PDF without lower reflection barrier.}
		\begin{figure}[H]
			\centering
			\includegraphics[width=0.7\textwidth]{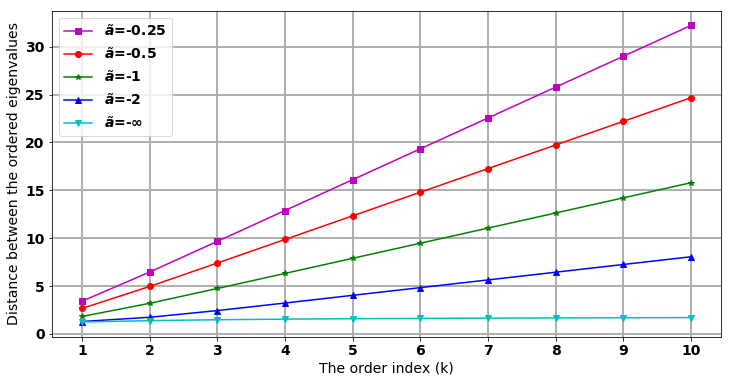}
			\caption{Distance between ordered eigenvalues for the PDE (\ref{specPDE}) with upper barrier $\tilde b = 1.5$ and different lower reflection barriers $\tilde a$.}\label{zero_fpt}
		\end{figure}
		\begin{figure}[H]
			\centering
			\includegraphics[width=0.7\textwidth]{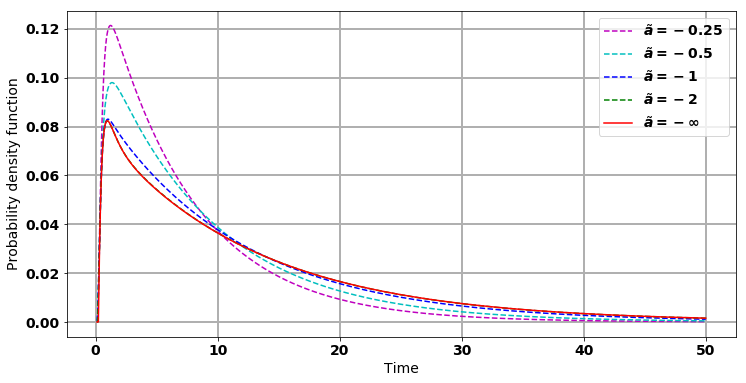}
			\caption{The probability density function for the first-passage-time of a standardized OU-process crossing the upper barrier $\tilde b = 1.5$ with lower reflection barrier $\tilde a$.}\label{pdf_fpt}
		\end{figure}
	\end{example}
}

\begin{corollary}\label{OUhitformula}
	The analytic form of the FPT survival function of the OU-process (\ref{homoOUeqn}) is given by
	\begin{align}
	\bar{F}_{\tau_{{X}, {b}}}(t, x) = \sum\limits_{k = 1}^{\infty} c_k \e^{-\lambda \alpha_k  t } \mathscr{H}_{\alpha_k}\left( -\sqrt{\frac{\lambda}{\sigma^2}}\left( x-\frac{\mu}{\lambda} \right) \right), \label{homoOUFPT}
	\end{align}
	where $\mathscr{H}_{\alpha_k}(\cdot)$ is the Hermite function, and the $\alpha_k$'s are the ordered solutions to the equation
	\begin{align*}
	\mathscr{H}_{\alpha_k}\left( -\sqrt{\frac{\lambda}{\sigma^2}}\left( b-\frac{\mu}{\lambda} \right) \right) = 0.
	\end{align*}
	Furthermore, the coefficient $c_k$ is given by
	\begin{align*}
	c_k &= -\frac{1}{\alpha_k \cdot \partial_{ \alpha_k} \mathscr{H}_{ \alpha_k}\left( -\sqrt{\frac{\lambda}{\sigma^2}}\left( b-\frac{\mu}{\lambda} \right) \right)}.
	\end{align*}
\end{corollary}

\begin{proof}
	Based on the relationship between the survival functions of the homogeneous, respectively, the standardized OU-process in Lemma \ref{OUtransform}, one can obtain this result by substituting the parameters in Lemma \ref{OUtransform} into Theorem \ref{OUhit}. 
\end{proof}

In the following theorem, we show the absolute convergence of the  infinite series (\ref{homoOUFPT}) and the bound of the truncation error utilising Corollary \ref{OUhitformula}. 

\begin{theorem}\label{thm_error}
The infinite series in formula (\ref{homoOUFPT}) is absolutely convergent. As $K\rightarrow \infty$, the truncated series $\sum\limits_{k = 1}^{K} c_k \e^{-\lambda \alpha_k  t } \mathscr{H}_{\alpha_k}\left( -\sqrt{\frac{\lambda}{\sigma^2}}\left( x-\frac{\mu}{\lambda} \right) \right)$ has truncation error $O\left( \e^{ -2K \lambda t } \right)$. 
Moreover, the absolute value of the truncation error is bounded by
$$\epsilon(\alpha_K) = \frac{\exp\left(\frac{{x^\prime}^2 - {b^\prime}^2}{2}\right)}{\sqrt{2}|b^\prime|}\left[ \frac{\exp\left(-\lambda t \alpha_K\right)}{\alpha_K} + (1-\lambda t )\Gamma(0, \lambda t \alpha_K) \right],$$
	where $x^\prime = \sqrt{\frac{\lambda}{\sigma^2}}\left( x-\frac{\mu}{\lambda} \right)$, ${b^\prime} = \sqrt{\frac{\lambda}{\sigma^2}}\left( b-\frac{\mu}{\lambda} \right)$, and $\Gamma(a,x)$ is the upper incomplete Gamma function with parameter $a$.
\end{theorem}

\begin{proof}
	By \cite{Lebedev}, as $\alpha_k\rightarrow\infty$, which means $k\rightarrow\infty$, we have
	\begin{align*}
	\mathscr{H}_{\alpha_k}\left( -x^\prime \right) = 2^{\alpha_k + 1/2} \e^{{x^\prime}^2 - \alpha_k/2-1/4} \left( \frac{\alpha_k}{2} + \frac{1}{4} \right)^{\alpha_k/2 } 
	\cos\left( 2x^\prime \sqrt{\frac{\alpha_k}{2} + \frac{1}{4}} - \frac{\alpha_k\pi}{2} \right)\left[ 1 + O\left( \frac{1}{\sqrt{\alpha_k/2+1/4}} \right) \right].
	\end{align*}
	Hence, for a large enough $k \in \mathbb{N}$, we have 
	\begin{align*}
	c_k=& -\left[1 +  O\left( \frac{1}{\sqrt{\alpha_k/2+1/4}} \right)\right] \cdot \left[ \alpha_k 2^{\alpha_k + \frac{1}{2}} \exp\left(\frac{{b^\prime}^2}{2} -\frac{\alpha_k}{2}-\frac{1}{4}\right)\left( \frac{\alpha_k}{2}+\frac{1}{4} \right)^{\alpha_k/2} \right.\\
	&\left. \qquad \times \sin\left( 2{b^\prime} \sqrt{\frac{\alpha_k}{2}+\frac{1}{4}} - \frac{\alpha_k\pi}{2} \right) \left(\frac{\pi}{2} + \frac{{b^\prime}\alpha_k}{\sqrt{\alpha_k/2+1/4}} \right) \right]^{-1}.
	\end{align*}
	For $k\rightarrow\infty$, we obtain the asymptotic behaviour of $\alpha_k$ 
	\begin{align}
	\alpha_k = 2k + 1 + \frac{4{b^\prime}^2}{\pi^2} +\frac{2{b^\prime}}{\pi}\sqrt{ 4k + 3 + \frac{4{b^\prime}^2}{\pi^2} }. \label{eqn_alpha}
	\end{align}
Therefore, for large enough $K\in \mathbb{N}$, the exact truncation error of Equation (\ref{homoOUFPT}) is
	\begin{align*}
	\sum_{k=K}^\infty c_k \mathscr{H}_{\alpha_k}\left( -{x^\prime} \right) \e^{-\lambda \alpha_k t}
	=\sum_{k=K}^\infty -\exp\left(\frac{{x^\prime}^2 - {b^\prime}^2}{2}\right) \frac{\cos\left( 2{x^\prime} \sqrt{\alpha_k/2+1/4} - \frac{\alpha_k\pi}{2} \right) \e^{-\lambda \alpha_k t} }{\alpha_k \sin\left( 2{b^\prime} \sqrt{\alpha_k/2+1/4} - \frac{\alpha_k\pi}{2} \right) \left( \frac{\pi}{2} + \frac{{b^\prime}\alpha_k}{\sqrt{\alpha_k/2+1/4}} \right)}.
	\end{align*}
	Since large $\alpha_k$'s satisfy Equation
	(\ref{eqn_alpha}), we have
	\begin{align*}
	\left|\sin\left( 2{b^\prime} \sqrt{\frac{\alpha_k}{2}+\frac{1}{4}} + \frac{\alpha_k\pi}{2} \right)\right| = 1.
	\end{align*}
	Therefore, we have the asymptotic inequality 
	\begin{align*}
	&\sum_{k=K}^\infty \left| c_k \mathscr{H}_{\alpha_k}\left( -{x^\prime} \right) \e^{-\lambda \alpha_k t} \right|\\
	 &\leq  \sum_{k=K}^\infty  \frac{ \exp\left(\frac{{x^\prime}^2 - {b^\prime}^2}{2}-\lambda \alpha_k t\right) }{\alpha_k \left| \frac{\pi}{2} + \frac{{b^\prime}\alpha_k}{\sqrt{\alpha_k/2+1/4}} \right|} \leq \exp\left(\frac{{x^\prime}^2 - {b^\prime}^2}{2}\right) C_1 \sum_{k=K}^\infty \frac{ \exp\left(-\lambda \alpha_k t\right) }{\alpha_k^{3/2}}\\
	&\leq \exp\left(\frac{{x^\prime}^2 - {b^\prime}^2}{2}\right) C_1 \exp\left(-\lambda \alpha_K t\right) \sum_{k=K}^\infty \frac{ 1 }{\alpha_k^{3/2}}\leq \exp\left(\frac{{x^\prime}^2 - {b^\prime}^2}{2}\right) C_1C_2 \exp\left(-\lambda \alpha_K t\right) \sum_{k=K}^\infty \frac{ 1 }{k^{3/2}} \\
	&\leq \exp\left(\frac{{x^\prime}^2 - {b^\prime}^2}{2}\right) C_1C_2C_3 \exp\left(-\lambda \alpha_K t\right) =  O\left( \exp\left(-\lambda \alpha_K t\right) \right) = O\left( \e^{ -2K \lambda t } \right).
	\end{align*}
	Since $\alpha_K > 0$, we have 
	\begin{align*}
	&\sum_{k=K}^\infty \left| c_k \mathscr{H}_{\alpha_k}\left( -{x^\prime} \right) \e^{-\lambda \alpha_k t} \right|\\
	&\leq  \sum_{k=K}^\infty  \frac{ \exp\left(\frac{{x^\prime}^2 - {b^\prime}^2}{2}-\lambda \alpha_k t\right) }{\alpha_k \left| \frac{\pi}{2} + \frac{{b^\prime}\alpha_k}{\sqrt{\alpha_k/2+1/4}} \right|} \leq \exp\left(\frac{{x^\prime}^2 - {b^\prime}^2}{2}\right)\sum_{k=K}^\infty \frac{ \exp\left(-\lambda \alpha_k t\right) \sqrt{\alpha_k + 1} }{ {{\sqrt{2} |b^\prime|}\alpha_k^2} } \\
	& \leq \exp\left(\frac{{x^\prime}^2 - {b^\prime}^2}{2}\right) \int_{\alpha_K}^{\infty} \frac{ \exp\left(-\lambda \alpha t\right) ({\alpha + 1}) }{ {{\sqrt{2} |b^\prime|}\alpha^2} } \rd \alpha = \frac{\exp\left(\frac{{x^\prime}^2 - {b^\prime}^2}{2}\right)}{\sqrt{2}|b^\prime|}\left[ \frac{\exp\left(-\lambda t \alpha_K\right)}{\alpha_K} + (1-\lambda t )\Gamma(0, \lambda t \alpha_K) \right].
	\end{align*}
\end{proof}
Theorem \ref{thm_error} gives an upper bound for the truncation error. One can determine the terms to be kept in order to achieve a specific accuracy level with an explicit function. An example of the error and its upper bound can be found in Figure \ref{err_bnd}.
\begin{figure}[H]
	\centering
	\includegraphics[width=0.5\textwidth]{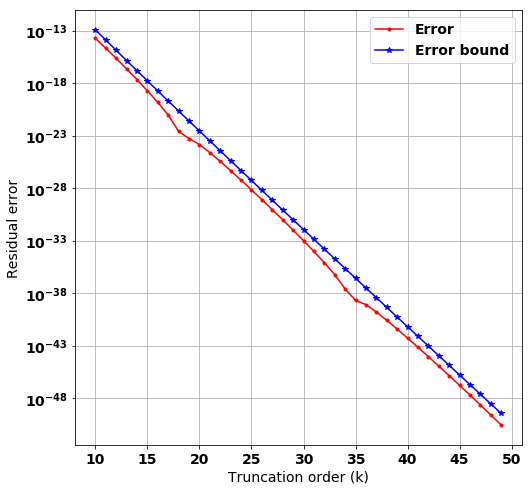}
	\caption{The log-scale plot between the truncation terms and the error when $x^\prime = 0$ and $b^\prime = 1$. }\label{err_bnd}
\end{figure}
To use the infinite series approximation for a probability density function, one needs to analyze how many terms one needs to keep in order to attain a certain precision level for the approximated distribution. Since one can transform the homogeneous OU-process barrier-crossing to a standardized OU-process barrier-crossing problem, see Lemma \ref{OUtransform}, we here study the truncation precision of the standardized OU-process. To this end, we proceed with Algorithm \ref{algor_color} in the appendix to study the relationship among the initial values of the process, barrier levels and the number of truncations. We plot the number of truncations required for various initial values and barrier levels in Figure \ref{color_plot1}, where the $\alpha$-zeros are taken in the interval $[0, 70]$. We observe that when the barrier level is far away from the initial value, the number of truncations required becomes smaller. Figure \ref{color_plot1} can be treated as a benchmark to determine how many terms one should truncate for a required quantile precision level. 
\begin{figure}[H]
	\centering
	\includegraphics[width=0.5\textwidth]{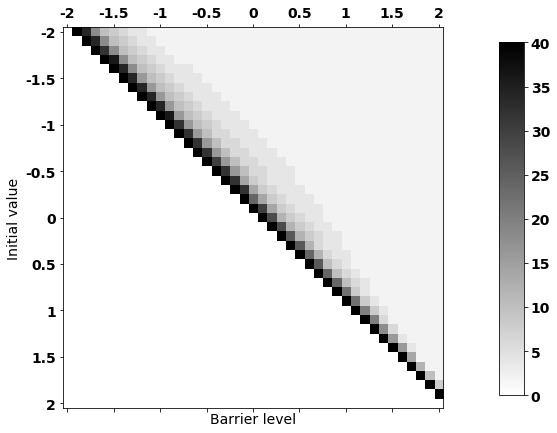}
	\caption{Relationship between the process initial values, barrier levels and the number of truncations. Here, the required quantile precision level is the median with relative error tolerance of $5\%$. }\label{color_plot1}
\end{figure}
\begin{remark}
	The $\alpha$-zeros can be obtained by the bisection method. However, these $\alpha$-zeros do not need to be obtained with high precision. Through the numerical test, we notice that if the $\alpha$-zeros are accurate up to $10^{-4}$, the approximation can be stable and reliable. 
\end{remark}
\begin{remark}
	For a barrier level $b$ which is larger than $5$, numerically solving the higher orders of $\alpha$-zeros (for $\alpha \geq 70$) becomes unstable. This is due to the value of $\mathscr{H}_{\alpha}(-b)$ becoming too large to be stored on a computer, leading to the overflow of the mantissa under double precision. For these cases, if the initial value is not next to the barrier level, one can truncate with fewer terms. This cannot cause larger errors due to the empirical results shown in Figure \ref{color_plot1}, in which one only needs to truncate with a few terms to reach a 5\% quantile precision. 
\end{remark}
\subsection{Tail behaviour of the FPT distribution for a homogeneous OU-process passing a constant barrier}
With the given infinite series representation in Equation (\ref{homoOUFPT}), we now analyze the property of the FPT distribution for a homogeneous process passing a constant barrier. With the method given on page 114 of \cite{Garethbook}, its tail behaviour can be characterized by the ``hazard rate function''. 
\begin{lemma}
	The distribution of the FPT of OU-process (\ref{homoOUeqn}) to a constant barrier $b$ is light-tailed, that is, the exponential moments exist up to the $\lambda \alpha_1$ order, where $\alpha_1$ is given in Corollary \ref{OUhitformula}. i.e. 
	$$\mathbb{E}\left[\e^{\theta \tau_{{X}, {b}}}\right]<\infty, \qquad \forall \theta < \lambda \alpha_1.$$ 
\end{lemma}

\begin{proof}
	We consider the hazard rate function given in \cite{Garethbook}, p. 114. The hazard rate function $r(t)$ for the FPT of OU-process (\ref{homoOUeqn}) is given by 
	\begin{align*}
	r(t) :=& \frac{-\partial_t \bar{F}_X(x, t;b) }{\bar{F}_X(x, t;b)} = 
	\frac{ \sum\limits_{k = 1}^{\infty} B_k \e^{-\lambda \alpha_k  t } }{\sum\limits_{k = 1}^{\infty} C_k \e^{-\lambda \alpha_k  t } }
	\end{align*}
	where 
	$B_k = \lambda \alpha_k c_k \mathscr{H}_{\alpha_k}( -\sqrt{\frac{\lambda}{\sigma^2}}( x-\frac{\mu}{\lambda}))$ and $C_k =c_k \mathscr{H}_{\alpha_k}( -\sqrt{\frac{\lambda}{\sigma^2}}( x-\frac{\mu}{\lambda}))$. 
	As shown in Remark 3.6 of \cite{Garethbook}, if $\lim_{t\rightarrow \infty} r(t) >0 $ exists, then the distribution is light-tailed and the exponential moments exist up to $\lim\inf_{t\rightarrow \infty} r(t)$. In our case, since $\lambda > 0$ is given by the definition of an OU-process and $\left\{\alpha_k\right\}_{k \in \mathbb{N}}$ are ordered positive solutions to the equation $$\mathscr{H}_{\alpha}\left( -\sqrt{\frac{\lambda}{\sigma^2}}\left( b-\frac{\mu}{\lambda} \right) \right) = 0$$ 
	with respect to $\alpha$, we have
	\begin{align*}
	\lim\limits_{t\rightarrow \infty} r(t) = \lim\limits_{t\rightarrow \infty} \frac{ \sum\limits_{k = 1}^{\infty} B_k \e^{-\lambda \alpha_k  t } }{\sum\limits_{k = 1}^{\infty} C_k \e^{-\lambda \alpha_k  t } } = \lim\limits_{t\rightarrow \infty} \frac{ B_1 + \sum\limits_{k = 2}^{\infty} B_k \e^{-\lambda (\alpha_k - \alpha_1)  t } }{C_1 + \sum\limits_{k = 2}^{\infty} C_k \e^{-\lambda (\alpha_k - \alpha_1)  t } } = \lambda \alpha_1. 
	\end{align*}
\end{proof}
\section{FPT transformation between an inhomogeneous and homogeneous OU-process crossing time-dependant barriers}\label{inhomo}
{In Section \ref{homo}, the eigenvalue expansion formulae is presented to compute the survival function of the FPT for a homogeneous OU-process passing a constant barrier. However, the homogeneous condition is usually too strong to model real-world events. For example, if we considered a stochastic process with a periodic feature, the time-homogeneity of a homogeneous process would be too rigid to allow the periodicity to be captured. In addition, barrier functions may be time-dependent, too, in reality. It is therefore natural and practically important to study the FPT for an inhomogeneous OU-process crossing a time-dependent barrier. 
\begin{definition}\label{hitmove}
	Let $(Z_t)_{t\geq0}$ be a continuous Markov process. The first-passage-time of $\left(Z_t\right)_{t\geq 0}$ with initial value $Z_0 = z$ to an upper time-dependent barrier $b(t)$, where $b(0) > z$, is defined by
	\begin{align}
	\mathscr{T}_{Z, b(t)}:=\inf\left\{ t\geq 0 : Z_t \geq b(t) \right\}.
	\end{align}
	The survival function of $\mathscr{T}_{Z, b(t)}$ is denoted by $\bar{F}_{\mathscr{T}_{Z, b(t)}}(t;z)$ and is given by
	\begin{align}
	\bar{F}_{\mathscr{T}_{Z, b(t)}}(t;z) = \mathbb{P}\left( \mathscr{T}_{Z, b(t)} > t \right). 
	\end{align}
\end{definition}
We focus in particular on the inhomogeneous OU-process, which is defined as follows.
\begin{definition}\label{inOUdef}
	Let $(W_t)_{t\geq 0}$ denote Brownian motion on the probability space $\left( \Omega, \mathscr{F}, \mathbb{P} \right)$. A solution $(Y_t)_{t\geq 0}$ to the stochastic differential equation 
	\begin{align}\label{OU}
	\rd  Y_t &= \left( \mu(t) - \lambda(t) Y_t \right)\rd  t + {\sigma(t)} \rd  W_t,
	\end{align}
	where $Y_0 = y \in \mathbb{R}$, is called an inhomogeneous Ornstein-Uhlenbeck process. For $\mu(t): \mathbb{R}^{+} \rightarrow \mathbb{R}$, $\lambda(t): \mathbb{R}^{+} \rightarrow \mathbb{R}^+$ and $\sigma(t): \mathbb{R}^{+} \rightarrow \mathbb{R}^+$ satisfying (i) $\left| \mu(t) - \lambda(t) y' \right| + \left| \sigma(t) \right| \leq C (1 + \left| y' \right|)$ for all $y'\in\R$ and $C\in \R$, and (ii) $\lambda(t)$ is bounded, $\forall t \geq 0$, the solution $(Y_t)_{t\geq 0}$ exists and is unique. 
\end{definition}
By Theorem 5.3.2 in \cite{oksendal}, the properties $(i)$ and $(ii)$ in Definition \ref{inOUdef} ensure that the SDE \ref{OU} has a unique $t$-continuous solution. A sufficient condition for $t$-continuity is for $\mu(t)$, $\lambda(t)$ and $\sigma(t)$ to be bounded. Next we show that the FPT distribution of an inhomogeneous OU-process crossing a time-dependent barrier is equivalent to the FPT distribution of a standardized OU-process crossing another time-dependent barrier.
	\begin{definition}\label{funcs}
		The mean-reverting scaling function $\alpha(t): \mathbb{R}^+ \rightarrow \mathbb{R}^+$, the shift function $\beta(t): \mathbb{R}^+ \rightarrow \mathbb{R}$, and the time-compensation function $\gamma(t): \mathbb{R}^+ \rightarrow \mathbb{R}^+$ are specified by:
		\begin{itemize}
			\item[a)] $\alpha(t)$, $\beta(t)$ and $\gamma(t) \in C^1(\mathbb{R^+})$ for $t >0$;
			\item[b)] $\gamma(t)$ is increasing for $t >0$;
			\item[c)] $\alpha(t)$, $\beta(t)$ and $\gamma(t)$ satisfy the ODE system
			\begin{align}\label{ode}
			\left\{
			\begin{array}{ll}
			\sigma(\gamma(t))\alpha(t)\sqrt{\gamma^\prime(t)} &= 1\\
			\lambda(\gamma(t))\gamma^\prime(t) - \frac{\alpha^\prime(t)}{\alpha(t)} &= 1\\
\beta^\prime(t)+\beta(t)-\alpha(t)\mu(\gamma(t))\gamma^\prime(t)&=0			
			\end{array}
			\right.
			\end{align}
			subject to the initial condition
			\begin{align}
			\alpha(0) = \alpha_0 \in \mathbb{R}^+ , \qquad \beta(0) = \beta_0 \in \mathbb{R}, \qquad \gamma(0) = 0
			\end{align} 
			where the constants $\alpha_0$ and $\beta_0$ are pre-determined.
		\end{itemize}
		The time-dependent parameters $\mu(t)$, $\lambda(t)$ and $\sigma(t)$ are specified in Definition \ref{inOUdef}.
	\end{definition}
\begin{lemma}\label{sufficient}
	Sufficient conditions for the uniqueness and existence of $\alpha(t)$, $\beta(t)$ and $\gamma(t)$ are ensured as follows:
	For $\mu(t)$, $\lambda(t)$ and $\sigma(t)$ given in Definition \ref{inOUdef}, if $\lambda(t) \in C^1\left( \mathbb{R}^+ \right)$ and $\sigma(t) \in C^2\left( \mathbb{R}^+ \right)$, then the ODE system (\ref{ode}) has a unique local solution with initial conditions $\alpha(0) = \alpha_0 \in \mathbb{R}^+$, $\beta(0) = \beta_0 \in \mathbb{R}$, and $\gamma(0) = 0$.
\end{lemma}

\begin{proof}
	The first two equations can be rearranged such that 
	\begin{align*}
	&\alpha(t) = \frac{1}{\sigma(\gamma(t))\sqrt{\gamma^\prime(t)}},&
	&\lambda(\gamma(t))\gamma^\prime(t) = 1 + \frac{\alpha^\prime(t)}{\alpha(t)}.&
	\end{align*}
	Substituting the first equation into the second, we obtain
	$$\frac{\gamma^{\prime\prime}(t)}{\gamma^\prime(t)} = 2-2\gamma^\prime(t)\left(\lambda(\gamma(t))+\frac{\sigma^\prime(\gamma(t))}{\sigma(\gamma(t))}\right).$$
	The equation can be written in the form
	$$\rd \ln(\gamma^\prime(t)) = 2 \rd t - 2\left(\lambda(\gamma(t))+\frac{\sigma^\prime(\gamma(t))}{\sigma(\gamma(t))} \right) \rd \gamma(t). $$
	Thus,
	$$\gamma\prime(t) = C\exp \left( 2t - 2\int^{\gamma(t)}_{0} \left(\lambda(y)+\frac{\sigma^\prime(y)}{\sigma(y)} \right) \rd y \right)$$
	where $C$ is a constant specified for a given initial condition, see \cite{ODEs}. 
	By the Picard-Lindel\"of theorem, see \cite{Lindelof}, the unique local solution $\gamma(t)$ is guaranteed, which in turn, by the first equation in (\ref{ode}), implies  the solution $\alpha(t)$ also exists and is unique. Since 
	$$\beta^\prime(t)+\beta(t)-\alpha(t)\mu(\gamma(t))\gamma^\prime(t)=0$$
	is a first-order linear ODE with respect to $\beta(t)$, whose solution is guaranteed to be unique, the ODE system (\ref{ode}) has a unique solution. 
\end{proof}

We note here that the necessary and sufficient conditions for the uniqueness and existence of $\alpha(t)$, $\beta(t)$ and $\gamma(t)$ are non-trivial. For $\alpha(t)$, $\beta(t)$ and $\gamma(t)$ specified as in Definition \ref{funcs}, an inhomogeneous OU-process is transformed into a standardized one as follows.

\begin{proposition}\label{inOUtransform}
	Consider the inhomogeneous OU-process $(Y_t)_{t\geq 0}$ given in Definition \ref{inOUdef}. Assume $\alpha(t), \beta(t)$ and $\gamma(t)$, in Definition \ref{funcs}, satisfy the sufficient conditions in Lemma \ref{sufficient}. Then the transformed process $(\alpha(t)Y_{\gamma(t)} - \beta(t))_{t\geq 0}$ is a standardized OU-process $(\widetilde{X}_t)_{t\geq 0}$, almost surely, with initial value $\widetilde{X}_0 = \alpha_0 y- \beta_0$. 
\end{proposition}

\begin{proof}
	The solution to the SDE (\ref{OU}) is given by
	\begin{align*}
	Y_t = \e^{-\int_{0}^{t}\lambda(u)\rd  u}\left[ y + \int_{0}^{t} \mu(s) \exp\left({\int_{0}^{s}\lambda(u)\rd  u}\right) \rd  s + \int_{0}^{t} {\sigma(s)} \exp\left({\int_{0}^{s}\lambda(u)\rd  u}\right) \rd  W_s \right].
	\end{align*}
	Since $\gamma(\cdot)$ satisfies the sufficient condition in Lemma \ref{sufficient}, $\gamma(\cdot)$ exists. Therefore,
	\begin{align*}
	Y_{\gamma(t)} = \e^{-\int_{0}^{\gamma(t)}\lambda(u)\rd  u} \left[ y + \int_{0}^{\gamma(t)} \mu(s) \exp\left({\int_{0}^{s}\lambda(u)\rd  u}\right) \rd  s +\int_{0}^{\gamma(t)} {\sigma(s)} \exp\left({\int_{0}^{s}\lambda(u)\rd  u}\right) \rd  W_s \right],
	\end{align*}
	and
	\begin{align*}
	Y_{\gamma(t)} = & \exp\left({-\int_{0}^{t}\lambda(\gamma(u))\gamma^\prime(u)\rd  u}\right)\left[ y + \int_{0}^{t} \mu(\gamma(s))\gamma^\prime(s) \right.
	\exp\left({\int_{0}^{s}\lambda(\gamma(u))\gamma^\prime(u)\rd  u}\right) \rd  s   \\
	&\hspace{4.5cm}\left.+\int_{0}^{t} {\sigma(\gamma(s))}\sqrt{\gamma^\prime(s)} \exp\left({\int_{0}^{s}\lambda(\gamma(u))\gamma^\prime(u)\rd  u}\right) \rd  W_s \right].
	\end{align*}
	It follows that
	\begin{align*}
	\rd Y_{\gamma(t)} =& \left[ \mu(\gamma(t))\gamma^\prime(t) - \lambda(\gamma(t))\gamma^\prime(t) Y_{\gamma(t)} \right] \rd t + \sigma(\gamma(t)) \sqrt{\gamma^\prime(t)} \rd W_t,
	\end{align*}
	and hence
	\begin{align*}
	\rd \widetilde{X}_t =& \left[ \alpha^\prime(t)Y_{\gamma(t)} - \beta^\prime(t) \right]\rd t + \alpha(t) \rd Y_{\gamma(t)}\\
	=& \left[ \alpha^\prime(t) \frac{\widetilde{X}_t + \beta(t)}{\alpha(t)} - \beta^\prime(t) + \alpha(t)\mu(\gamma(t))\gamma^\prime(t)-\alpha(t)\lambda(\gamma(t))\gamma^\prime(t) \frac{\widetilde{X}_t + \beta(t)}{\alpha(t)} \right] \rd t\\
	&\hspace{9cm}+\alpha(t) \sigma(\gamma(t)) \sqrt{\gamma^\prime(t)} \rd W_t \\
	=&-\widetilde{X}_t \rd t + \rd W_t.
	\end{align*}
	In the last step, Definition \ref{funcs} is used. Thus, $(\widetilde{X}_t)_{t \geq 0}$ is a standardized OU-process. 
\end{proof}

Now we are in the position to present the main theorem that links the FPT distribution functions of the inhomogeneous and the standardized OU-processes. 

\begin{theorem}\label{inOUtransSOU}
	Let $(Y_t)_{t\geq 0}$ be the inhomogeneous OU-process in Definition \ref{inOUdef}, and assume that the ODE (\ref{ode}) has a unique solution. Then,
	\begin{align}
	\bar{F}_{\mathscr{T}_{Y, b(t)}}(t;y) = \bar{F}_{\mathscr{T}_{\widetilde{X}, g(t)}}\left(\gamma^{-1}(t);\tilde{x}\right)
	\end{align}
	where $(\widetilde{X}_t)_{t\geq 0}$ is a standardized OU-process with initial value $\widetilde{X}_0 = \tilde{x} = \alpha_0 y - \beta_0$, and $$g(t)=\alpha(t)b(\gamma(t)) - \beta(t).$$ 
	An equivalent statement is to say $\mathscr{T}_{Y, b(t)}$ and $\gamma(\mathscr{T}_{\widetilde{X}, g(t)})$ are equal in distribution.
\end{theorem}

\begin{proof}
	First, we show $\mathscr{T}_{Y, b(t)}$ and $\gamma\left(\mathscr{T}_{\widetilde{X}, \alpha(t)b(\gamma(t)) - \beta(t)}\right)$ are equal in distribution. We have that 
	\begin{align*}
	\mathscr{T}_{Y, b(t)} =& \inf\left\{ t>0: Y_t \geq b(t) \right\} = \inf\left\{ \gamma(t)>0: Y_{\gamma(t)} \geq b(\gamma(t)) \right\}\\
	=& \inf\left\{ \gamma(t)>0: \alpha(t)Y_{\gamma(t)} - \beta(t) \geq \alpha(t)b(\gamma(t)) - \beta(t) \right\}.
	\end{align*}
	Since $\gamma(\cdot)$ is monotone, non-decreasing and positive, we deduce
	\begin{align*}
	\mathscr{T}_{Y, b(t)} =& \gamma\left( \inf\left\{ t>0: \alpha(t)Y_{\gamma(t)} - \beta(t) \geq \alpha(t)b(\gamma(t)) - \beta(t) \right\} \right).
	\end{align*}
	By Proposition \ref{inOUtransform}, we know that the process $(\alpha(t)Y_{\gamma(t)} - \beta(t))_{t\geq 0}$ has the law of a standardised OU-process. Therefore,
	\begin{align*}
	\mathscr{T}_{Y, b(t)} = \gamma\left( \inf\left\{ t>0: \alpha(t)Y_{\gamma(t)} - \beta(t) \geq \alpha(t)b(\gamma(t)) - \beta(t) \right) \right) = \gamma\left(\mathscr{T}_{\widetilde{X}, \alpha(t)b(\gamma(t)) - \beta(t)}\right).
	\end{align*}
	Then, it follows that
	$\bar{F}_{\mathscr{T}_{Y, b(t)}}(t;x) = \mathbb{P}\left( \mathscr{T}_{Y, b(t)} >t\, | \,  Y_0 = x \right) = \mathbb{P}\left( \gamma\left(\mathscr{T}_{\widetilde{X}, \alpha(t)b(\gamma(t)) - \beta(t)}\right) >t\, | \,  \widetilde X_0 =\widetilde x \right)
	= \mathbb{P}\left( \mathscr{T}_{\widetilde{X}, \alpha(t)b(\gamma(t)) - \beta(t)} >\gamma^{-1}(t)\, | \,  \widetilde X_0 =\widetilde x \right) = \bar{F}_{\mathscr{T}_{\widetilde{X}, g(t)}}\left(\gamma^{-1}(t);\widetilde{x}\right)$.
\end{proof}

\begin{example}[\textbf{Seasonal trend}]
		One example is to apply a seasonality function to the mean-reverting level function $\mu(t)$. Here we show how we can utilize Theorem \ref{inOUtransSOU} to transform the problem of an inhomogeneous OU-process hitting a constant barrier to the one of a standardized OU-process hitting a periodic barrier. We consider the inhomogeneous OU-process $(Y_t)_{t \geq 0}$, parametrized by $\mu(t) = A \sin\left( \theta t + \varphi \right)$, $\lambda(t) = \lambda$ and $\sigma(t) = \sigma$, with initial value $Y_0 = y$, where $A$, $\theta$, $\varphi \in \mathbb{R}$ and $\lambda, \sigma > 0$. The constant barrier is denoted by $b$. The mean-reverting scaling function $\alpha(t)$ and the time-compensation function $\gamma(t)$ are given by $\alpha(t) = \sqrt{\lambda}/\sigma$ and $\gamma(t) = t/\lambda$.
		Then $\beta(t)$ satisfies 
\begin{align}\label{beta_ode}
\beta(t) = \mu\left( \frac{t}{\lambda} \right) \frac{1}{\sigma \sqrt{\lambda}}-\beta^\prime(t).
\end{align}
		The associated ODE (\ref{beta_ode}), in this particular case, has the unique solution
		\begin{align}
		\beta(t) = B \rm{e}^{-t} + \frac{A \sqrt{\lambda}}{\sigma \sqrt{\lambda^2 + \theta^2}} \sin\left( \frac{\theta}{\lambda}t + {\varphi} - \arctan\left( \frac{\theta}{\lambda} \right) \right),
		\end{align}
		where $B$ is a constant so to match the initial condition. For convenience, we let $B=0$ by imposing the initial condition
		$$\beta(0) =  \frac{A \sqrt{\lambda}}{\sigma \sqrt{\lambda^2 + \theta^2}} \sin\left( {\varphi} - \arctan\left( \frac{\theta}{\lambda} \right) \right),$$
		which, by Theorem \ref{inOUtransSOU}, means that the standardized OU-process $(\widetilde{X}_t)_{t\geq  0}$ starts from
		$$\widetilde{X}_0 = \frac{\sqrt{\lambda}}{\sigma} X_0 - \frac{A \sqrt{\lambda}}{\sigma \sqrt{\lambda^2 + \theta^2}} \sin\left( {\varphi} - \arctan\left( \frac{\theta}{\lambda} \right) \right).$$
		Then we have a particular solution for $\beta(t)$ given by
		\begin{align*}
		\beta(t) = \frac{A \sqrt{\lambda}}{\sigma \sqrt{\lambda^2 + \theta^2}} \sin\left( \frac{\theta}{\lambda}t + {\varphi} - \arctan\left( \frac{\theta}{\lambda} \right) \right).
		\end{align*}
		By Theorem \ref{inOUtransSOU}, we can now calculate the probability of a standardized OU-process, with initial value
		$$\widetilde{X}_0 = \widetilde{x} = \frac{x\sqrt{\lambda}}{\sigma} - \frac{A \sqrt{\lambda}}{\sigma \sqrt{\lambda^2 + \theta^2}} \sin\left( {\varphi} - \arctan\left( \frac{\theta}{\lambda} \right) \right),$$ 
		crossing a periodic barrier 
		$$g(t) = \frac{b\sqrt{\lambda}}{\sigma} - \frac{A \sqrt{\lambda}}{\sigma \sqrt{\lambda^2 + \theta^2}} \sin\left( \frac{\theta}{\lambda}t + {\varphi} - \arctan\left( \frac{\theta}{\lambda} \right) \right). $$	
\end{example}
\section{Multiple crossings of an inhomogeneous OU-process}\label{multi}
For an inhomogeneous OU-process with parameter functions in ${C}^1(\mathbb{R}^+)$, one can transform its barrier-crossing problem to one involving the standardized OU-process and a time-dependent barrier function. The time-dependent barrier can be approximated by a piece-wise constant function. This is due to the fact that any continuous function can be approximated with a piece-wise constant function to arbitrary accuracy given a sufficiently large number of partitions. The problem thus reduces to a multiple-crossing problem for a standardized OU-process. We call this scheme the {\it transformation method}.

Alternatively, one may directly use the piecewise constant approximation for the parameter functions of the inhomogeneous OU-process. This alternative method leads to a multiple-crossing problem for a locally-homogeneous OU-process. We call this scheme the {\it direct approximation method}.

The transformation method modifies the inhomogeneous OU-process to a global standardized OU-process by solving an ODE system given in Definition \ref{funcs}. After the transformation, it uses piece-wise constant functions to approximate the new time-varying barrier. This scheme requires further conditions to be satisfied, suchlike continuity of the parameter functions, for the transformation to be well-defined. In addition, solving the ODE system can be difficult. The second method does not rely on such a transformation. However, it results in a locally homogeneous OU-process with piece-wise constant barriers, where the time steps for the barriers and OU-parameters may not necessarily match. 
\begin{figure}[H]
	\centering
	\includegraphics[width=0.475\textwidth]{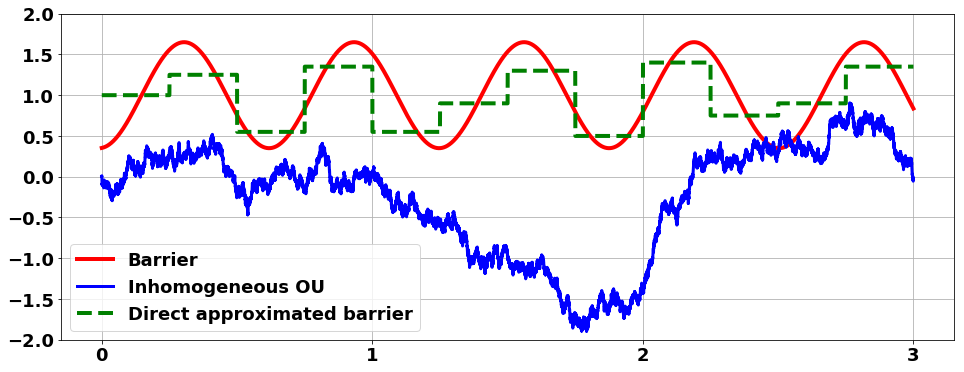}
	\includegraphics[width=0.475\textwidth]{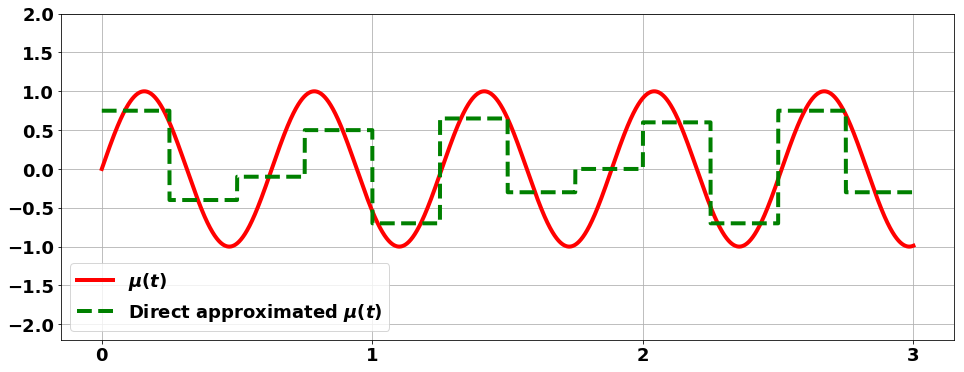}
	\includegraphics[width=0.475\textwidth]{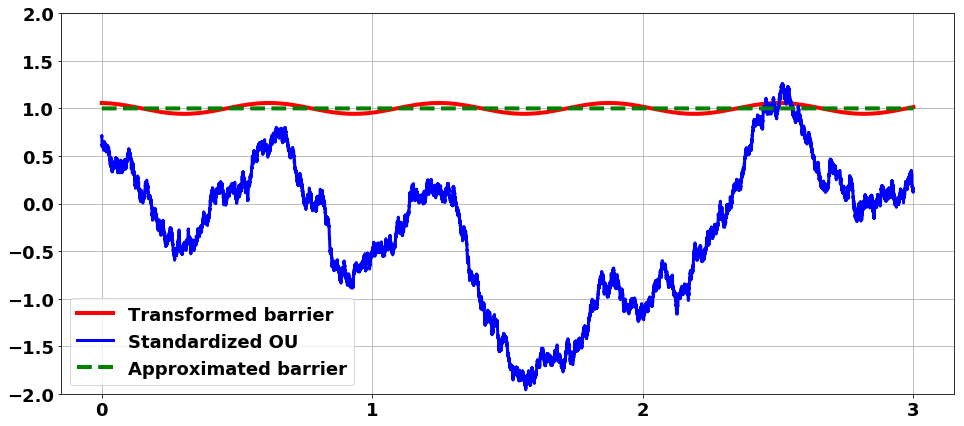}
	\caption{An example where the {\it transformation method} is advantageous and the piece-wise-constant approximation is applied. The inhomogeneous OU-process crossing the time-dependent barrier $b(t)=1+0.65\sin\left( 10t+\arctan(10) \right)$ is shown in the upper panel. The parameter functions of the OU-process are $\mu(t)=\sin(10t)$ and $\lambda=\sigma=1$ as shown in the second panel. The original problem can be transformed to a standardized OU-process with smoother time-dependent barrier, see the lower panel.}\label{OUpath1}
\end{figure}
In applications, the method one should select is decided on a case by case basis. In Figure \ref{OUpath1}, we simulate a time-inhomogeneous OU-process with a time-dependent barrier. The application of the first method can offset time-dependencies from the parameters and the barrier. The direct approximation method will lead to a higher approximation error. In order to reach the same level of accuracy, one may have to approximate using more time segments, which complicates the barrier-crossing problem. 
\begin{figure}[H]
	\centering
	\includegraphics[width=0.475\textwidth]{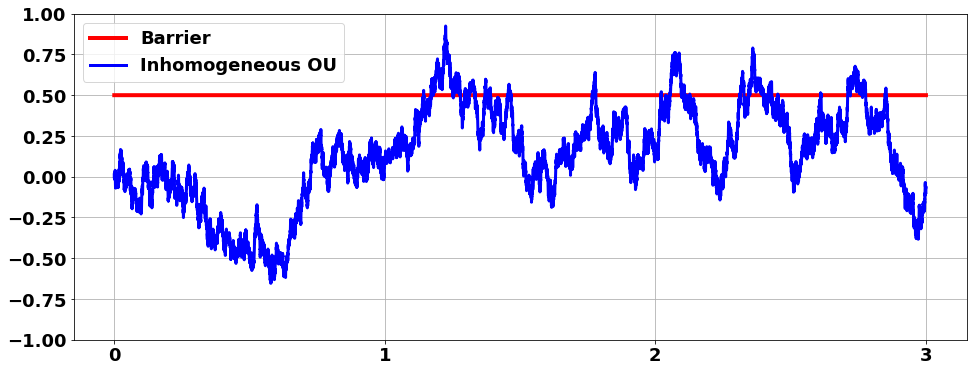}
	\includegraphics[width=0.475\textwidth]{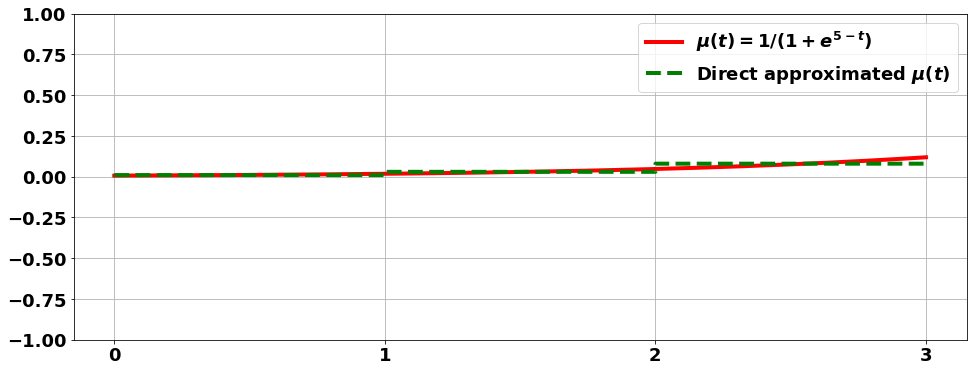}
	\includegraphics[width=0.475\textwidth]{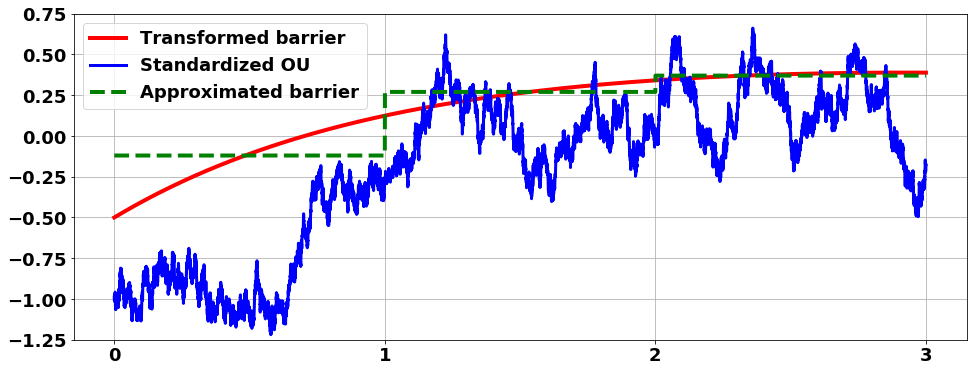}
	\caption{An example where the piecewise-constant approximation is inefficient and the {\it transformation method} is disadvantageous. The inhomogeneous OU-process crossing the constant barrier $b(t)=0.5$ is shown in the upper panel. The parameter functions of the OU-process are $\mu(t)=1/(1+\e^{5-t})$ and $\lambda=\sigma=1$ as shown in the second panel. The original problem is transformed to a standardised OU-process with a steeper time-dependent barrier, see the lower panel. }\label{OUpath2}
\end{figure}
However, this does not mean that the first method is always better than the direct approximation approach. For example in Figure \ref{OUpath2}, with the same number of discretizations, the transformation method leads to a higher approximation error.

In general, if any of the three OU parameter functions is not in ${C}^1(\mathbb{R}^+)$, one should use the direct approximation method. When applying the transformation method, the transformed barrier function $g(t) = \alpha(t)b(\gamma(t)) - \beta(t)$ can be written as
$$g(t) = \e^{-t}\left[ \alpha_0 \e^{\int_{0}^{\gamma(t)} \lambda(s) \rd s} b(\gamma(t)) - \beta_0 - \alpha_0 \beta_0 \int_{0}^{\gamma(t)} \e^{\int_{0}^{s} \lambda(s) \rd s} \mu(s) \rd s \right],$$
where $\gamma(t)$ is obtained from the equation
$$\alpha_0 \e^{\int_{0}^{\gamma(t)} \lambda(s) \rd s - t} \sigma(\gamma(t))\sqrt{\gamma^\prime(t)} = 1.$$
Although it is difficult to devise a general principle to select the method to be adopted, we can provide a rule for some special cases as, e.g., in the next proposition.
\begin{proposition}\label{convexity}
	Assume that the inhomogeneous OU-process in Definition \ref{inOUdef} has coefficients $\sigma(t) = \sigma$, $\lambda(t) = \lambda$, $\mu(t) \in {C}^2\left([t_1, t_2]\right)$, and consider a constant barrier $b(t) = b$. The transformed barrier function is denoted by $g(t) = \alpha(t) b(\gamma(t))- \beta(t)$, where $\alpha(t), \beta(t)$ and $\gamma(t)$ are defined in Definition \ref{funcs}. If $\text{sgn}(g^\prime(\lambda t))=\text{sgn}(\mu^\prime(t))$, $\text{sgn}(g^{\prime\prime}(\lambda t))=\text{sgn}(\mu^{\prime\prime}(t))$ and further
	$\mu(t)$ satisfies 
	$$ \frac{\inf\limits_{t \in [\lambda t_1, \lambda t_2]} {\beta_0} \left| \left[ \mu(\frac{t}{\lambda})    -  \sigma\sqrt{\lambda} \e^{-t} - \int_{0}^{t} \e^{s-t} \mu\left( \frac{s}{\lambda} \right) \rd s   - \mu^\prime(\frac{t}{\lambda}) \frac{1}{\lambda}  \right] \right| }{\inf\limits_{t \in [t_1, t_2]}  \left| \mu^{\prime\prime}(t) \right|} \lessgtr \sigma\sqrt{\lambda},
	$$
	for all $t \in [t_1, t_2]$, then $g'(t)\lessgtr\mu'(t)$, and it is efficient to use the transformation method (for $<$), respectively, the direct approximation method (for $>$).
\end{proposition}
\begin{proof}
	See Appendix \ref{AppC}.
\end{proof}
\begin{example}
	Consider the inhomogeneous OU-process in Definition \ref{inOUdef} with $\lambda(t) = \sigma(t) = 1$ and $\mu(t) = -\e^{-t}$ crossing the constant barrier $b(t) = 1$ in the interval $[1,3/2]$. By the transformation method, we have $\alpha(t) = 1$, $\gamma(t) = t$ and $\beta(t) = \e^{-t}(1-t)$. Therefore for $t\in[1,3/2]$,
	\begin{align*}
	&\mu(t) = -\e^{-t},& &g(t) = 1 + \text{\e}^{-t}(t-1),&\\
	&\mu^\prime(t) = \text{\e}^{-t}>0,& &g^\prime(t) = \text{\e}^{-t}(2-t)>0,&\\
	&\mu^{\prime\prime}(t) = -\text{\e}^{-t}<0,& &g^{\prime\prime}(t)=\e^{-t}(t-3)<0.&
	\end{align*}
	This shows that both $\mu(t)$ and $g(t)$ are concave monotone increasing functions in the given domain. We have 
	\begin{align*}
	\inf\limits_{t \in [ 1,3/2]} | g^{\prime\prime}(t) | &= \inf\limits_{t \in [ 1,3/2]} |\e^{-t}(t-3)| = \frac{3}{2}\text{\e}^{-\frac{3}{2}},\\ 
	\inf\limits_{t \in [ 1,3/2]}  | \mu^{\prime\prime}(t)| &= \inf\limits_{t \in [ 1,3/2]} | \text{\e}^{-t} | = \e^{-\frac{3}{2}}.
	\end{align*}
	Since
	$$\frac{\inf\limits_{t \in [ 1,3/2]} | g^{\prime\prime}(t) |}{\inf\limits_{t \in [ 1,3/2]} | \mu^{\prime\prime}(t) |}>1,$$
	it is more efficient to use the direct approximation method. 
\end{example}
Next we investigate the case where $\mu(t)$, $\lambda(t)$ and $\sigma(t)$, the parameter functions of the inhomogeneous OU-process $(Y_t)_{t \geq 0}$, and the barrier function $b(t)$ are c\`adl\`ag piece-wise constant functions. Let the parameter functions be specified by 
\begin{align*}
&\mu(t) = \sum_{i = 1}^{N^{(\mu)}} \mu_i \mathbb{1}\left({t \in [t^{(\mu)}_{i-1}, t^{(\mu)}_i)}\right),&
&\lambda(t) = \sum_{i = 1}^{N^{(\lambda)}} \lambda_i \mathbb{1}\left({t \in [t^{(\lambda)}_{i-1}, t^{(\lambda)}_i)}\right),&\\
&\sigma(t) = \sum_{i = 1}^{N^{(\sigma)}} \sigma_i \mathbb{1}\left({t \in [t^{(\sigma)}_{i-1} t^{(\sigma)}_i)}\right),&
&b(t) = \sum_{i = 1}^{N} b_i \mathbb{1}\left({t \in [t_{i-1}, t_i)}\right),&
\end{align*}
for all $i=1, 2, ..., N$,	where $\lambda_i, \sigma_i \in \mathbb{R}^+$, $\mu_i, b_i \in \mathbb{R}$ and $b_0 > Y_0$. Here, we consider a finite-time horizon where $t^{(\mu)}_{N^{(\mu)}} = t^{(\lambda)}_{N^{(\lambda)}} = t^{(\sigma)}_{N^{(\sigma)}} = t_{N}$, and $\mathbb{1}\left(\cdot\right)$ denotes the indicator function. We study the following probabilities:
\begin{align}
&\mathbb{P}\left( M_{t_{0}, t_1}^{Y}<b_1, M_{t_{1}, t_2}^{Y}<b_2,\ldots, M_{t_{N-1}, t_N}^{Y}<b_N \right), \label{joint_dist}\\
&\mathbb{P}\left( M_{t_{0}, t_1}^{Y}\geq b_1, M_{t_{1}, t_2}^{Y}\geq b_2,\ldots, M_{t_{N-1}, t_N}^{Y}\geq b_N \right),\label{joint_survival}
\end{align}
where 
$$M_{t_{i-1}, t_i}^{Y} = \sup\limits_{t \in ( t_{i-1}, t_i ]} Y_t.$$
In particular, the probability (\ref{joint_dist}) is equal to the probability that the FPT of this inhomogeneous OU-process is larger than $t_N$. Expression (\ref{joint_survival}) is the probability that the inhomogeneous OU-process crosses the barrier in each interval, i.e. the multiple crossing (joint) probability. We may consider the following discretization schemes:
\begin{itemize}
	\item[1)]
	Matching time-discretization for $\mu(t)$, $\lambda(t)$, $\sigma(t)$ and $b(t)$, i.e. $t^{(\mu)}_i = t^{(\lambda)}_i = t^{(\sigma)}_i = t_i$ for all $i = 0, 1, \ldots, N$.
	\item[2)]
	Matching time-discretization for $\mu(t)$, $\lambda(t)$ and $\sigma(t)$ only, i.e. $t^{(\mu)}_i = t^{(\lambda)}_i = t^{(\sigma)}_i$ for all $i = 0, 1, \ldots, N^{(\mu)}$.
	\item[3)]
	Non-matching time-discretizations for any of $\mu(t)$, $\lambda(t)$, $\sigma(t)$ and $b(t)$. 
\end{itemize}
One can show that the probabilities (\ref{joint_dist}) and (\ref{joint_survival}) in the last two cases can be further reduced to the first case by utilizing Theorem \ref{main_thm}. Therefore, in what follows, we will focus on the first case unless specified otherwise, and we shall show the reduction methods from the case 2) and 3) to case 1) in Section \ref{nonmatching}.
\subsection{Joint distribution and multivariate survival functions for multiple maxima of a continuous Markov process in consecutive intervals}\label{inhomoFPT}

We begin with a theorem for a continuous Markov process. We recall the definition of a Markov process, see for instance \cite{bingham}. In this section, we introduce the filtered probability space $(\Omega, \mathscr{F},(\mathscr{F}_t), \mathbb{P})$ and an adapted Markov process $(Z_t)_{t\ge 0}$. We write $\sigma(Z_s:s\le t)$ for the natural filtration of the Markov process $(Z_t)$, where $\sigma(Z_s:s\le t)\subseteq\F_t$.
\begin{lemma}\label{markov_ind}
If $(Z_t)_{t\ge0}$ is a Markov process, then $\mathbb{P}(A\, \cap\, B\,\vert\, Z_t ) = \mathbb{P}(A \,\vert\, Z_t )\mathbb{P}(B \,\vert\, Z_t )$ for all $A \in \sigma (Z_u: u \ge t)$ and $B \in \sigma (Z_u: u \le t)$.  It follows that $\mathbb{P}(A \,\vert\, B, Z_t ) = \mathbb{P}(A \,\vert\, Z_t )$.
\end{lemma}
\begin{proof}
	This is straightforward and shown, e.g., in \cite{bingham}.
\end{proof}

We prove Theorem \ref{main_thm} below by use of the conditional independence property. We consider the time steps $0 = t_0 < t_1 < \cdots < t_N = T$, and denote the barrier level in the interval $[t_{i-1}, t_i)$ by $b_i$. In Theorem \ref{main_thm}, we calculate the joint distribution function, or the survival function, of the maximum of a continuous Markov process in each interval. Here, we write $M_{t_{i-1}, t_i}:=\sup\limits_{t\in(t_{i-1}, t_i]} Z_t$.

\begin{theorem}\label{main_thm}
	Let $b_1, \ldots, b_N\in \mathscr{D} := \text{Dom}(Z_t)$. The joint distribution and survival functions of the maxima of a continuous Markov process $(Z_t)_{t \geq 0}$ in consecutive intervals are given, respectively, by
	\begin{align*}
	&\mathbb{P}\left( M_{t_{0},t_1} < b_1, \cdots, M_{t_{N-1},t_N} < b_N\, | \,  Z_0 = z_0 \right) \\
	&=\int_{\mathscr{D}} \psi_1\left(t_{0}, t_1,  z_{0},  z_{1}, b_1 \right) \cdots \int_{\mathscr{D}} \psi_{N-1}\left( t_{N-2}, t_{N-1},  z_{N-2},  z_{N-1}, b_{N-1} \right)q\left( t_{N-1}, t_N, Z_{N-1} \right) \rd z_{N-1} \cdots \rd z_1,\\\\
	&\mathbb{P}\left( M_{t_{0},t_1} \geq b_1, \cdots, M_{t_{N-1},t_N} \geq b_N\, | \,  Z_0 = z_0 \right) \\
	&= \int_{\mathscr{D}} \kappa_1\left(t_{0}, t_1,  z_{0},  z_{1}, b_1 \right) \cdots \int_{\mathscr{D}} \kappa_{N-1}\left( t_{N-2}, t_{N-1},  z_{N-2},  z_{N-1}, b_{N-1} \right) 
	\bar{q}\left( t_{N-1}, t_N, Z_{N-1} \right) \rd z_{N-1} \cdots \rd z_1,
	\end{align*}
	where 
	\begin{align*}
		\psi_{i}(t_{i-1}, t_i, z_{i-1}, z_i, b_i)&=\mathbb{P}\left( M_{t_{i-1},t_i} < b_i\, | \,  Z_{t_{i-1}} = z_{i-1}, Z_{t_i}=z_i \right) p(t_{i-1}, t_i,z_{i-1},z_i),\\
		\kappa_{i}(t_{i-1}, t_i, z_{i-1}, z_i, b_i)&=\mathbb{P}\left( M_{t_{i-1},t_i} \geq b_i\, | \,  Z_{t_{i-1}} = z_{i-1}, Z_{t_i}=z_i \right) p(t_{i-1}, t_i, z_{i-1}, z_i),\\
		q\left( t_{N-1}, t_N, Z_{N-1} \right)&=\mathbb{P}\left( M_{t_{N-1},t_N} < b_N\, | \,  Z_{t_{N-1}} = z_{N-1} \right),\\
		\bar{q}\left( t_{N-1}, t_N, Z_{N-1} \right) &= 1 - {q}\left( t_{N-1}, t_N, Z_{N-1} \right).
	\end{align*}
	Here $p(t_{i-1}, t_i,z_{i-1},z_i)$, for $i=1,\ldots, N-1$, is the transition density function of the process $(Z_t)_{t \geq 0}$ from state $z_{i-1}$ at time $t_{i-1}$ to state $z_{i}$ at time $t_{i}$. 
\end{theorem}

\begin{proof}
	We show the proof for the joint distribution function. The proof for the joint survival function is similar, since we also utilize the Markov conditional independence property. We proceed with a proof by induction. 
	\begin{itemize}
		\item[(1)] Case $N = 2$: We know that
		\begin{align*}
		&\mathbb{P}\left( M_{t_{0},t_1} < b_1, M_{t_{1},t_2} < b_2\, | \,  Z_0 = z_0 \right) \\
		&\hspace{0cm}= \int_{\mathscr{D}} \mathbb{P}\left( M_{t_{0},t_1} < b_1, M_{t_{1},t_2} < b_2\, | \,  Z_{t_1} =  z_1, Z_0 = z_0  \right) \mathbb{P}_{z_0}\left( Z_{t_1} \in \rd z_1\, | \,  Z_0 = z_0 \right).
		\end{align*}
		Since 
		$\{ M_{t_{0},t_1} < b_1\} \in \sigma( Z_s: s\leq t_1)$
		and
		$\{ M_{t_{1},t_2} < b_2\} \in \sigma( Z_s: s\geq t_1)$,
		we have
		\begin{align*}
		&\mathbb{P}\left( M_{t_{0},t_1} < b_1, M_{t_{1},t_2} < b_2\, | \,  Z_0 = z_0 \right) \\
		&= \int_{\mathscr{D}} \mathbb{P}\left( M_{t_{0},t_1} < b_1\, | \,  Z_{t_1} = z_1, Z_0 = z_0 \right) \mathbb{P}\left( M_{t_{1},t_2} < b_2\, | \,  Z_{t_1} = z_1, Z_0 = z_0 \right)\mathbb{P}\left( Z_{t_1} \in \rd z_1\, | \,  Z_0 = z_0 \right),
		\end{align*}
		by Lemma \ref{markov_ind}. Then by the Markov property, we obtain
		\begin{align*}
		&\mathbb{P}\left( M_{t_{0},t_1} < b_1, M_{t_{1},t_2} < b_2\, | \,  Z_0 = z_0 \right)=\int_{\mathscr{D}} \psi_1\left(t_{0}, t_1,  z_{0},  z_{1}, b_1 \right) q\left( t_{1}, t_2, Z_{1} \right) \rd z_1,
		\end{align*}
		which is the case $N=2$ in Theorem \ref{main_thm}. 
		
		\item[(2)]Consider Theorem \ref{main_thm} for $N=K$ such that 
		\begin{align*}
		&\mathbb{P}\left( M_{t_{0},t_1} < b_1, M_{t_{1},t_2} < b_2, \cdots, M_{t_{K-1},t_K} < b_K\, | \,  Z_0 = z_0 \right) \\
		&=\int_{\mathscr{D}} \psi_1\left(t_{0}, t_1,  z_{0},  z_{1}, b_1 \right) \cdots \int_{\mathscr{D}} \psi_{K-1}\left( t_{K-2}, t_{K-1},  z_{K-2},  z_{K-1}, b_{K-1} \right)q\left( t_{K-1}, t_K, Z_{K-1} \right) \rd z_{K-1} \cdots \rd z_1.
		\end{align*}
		Now consider the case when $N = K+1$. We have 
		\begin{align*}
		&\mathbb{P}\left( M_{t_{0},t_1} < b_1, M_{t_{1},t_2} < b_2, \cdots, M_{t_{K-1},t_K} < b_K, M_{t_{K},t_{K+1}} < b_{K+1}\, | \,  Z_0 = z_0 \right) \\
		&= \int_{\mathscr{D}} \mathbb{P}\left( M_{t_{0},t_1} < b_1, M_{t_{1},t_2} < b_2, \cdots, M_{t_{K},t_{K+1}} < b_{K+1}\, | \,  Z_{t_1} = z_1, Z_0 = z_0 \right)\mathbb{P}_{z_0}\left( Z_{t_1} \in \rd  z_1\, | \,  Z_0 = z_0 \right).
		\end{align*}
		Then, by Lemma \ref{main_thm}, since 
		$\{ M_{t_{0},t_1} < b_1\} \in \sigma( Z_s: s\leq t_1)$
		and 
		$\{ M_{t_{1},t_2} < b_2, \cdots, M_{t_{K},t_{K+1}} < b_{K+1} \} \in \sigma( Z_s: s\geq t_1)$,
		we have 
		\begin{align*}
		&\mathbb{P}\left( M_{t_{0},t_1} < b_1, M_{t_{1},t_2} < b_2, \cdots, M_{t_{K},t_{K+1}} < b_{K+1}\, | \,  Z_{t_1} = z_1, Z_0 = z_0 \right)\\ 
		&= \mathbb{P}\left( M_{t_{0},t_1} < b_1\, | \,  Z_{t_1} = z_1, Z_0 = z_0 \right)\mathbb{P}\left( M_{t_{1},t_2} < b_2,  \cdots, M_{t_{K},t_{K+1}} < b_{K+1}\, | \,  Z_{t_1} =  z_1, Z_0 = z_0 \right).
		\end{align*}
		For $N=K$, by the Markov property, we have
		\begin{align*}
		&\mathbb{P}\left( M_{t_{1},t_2} < b_2, M_{t_{2},t_3} < b_3, \cdots, M_{t_{K},t_{K+1}} < b_{K+1}\, | \,  Z_{t_1} = z_1 \right)\\
		&\hspace{1.5cm}=\int_{\mathscr{D}} \psi_1\left(t_{1}, t_2,  z_{1},  z_{2}, b_2 \right) \cdots \int_{\mathscr{D}} \psi_K\left( t_{K-1}, t_{K},  z_{K1},  z_{K}, b_{K} \right) q\left( t_{K}, t_{K+1}, Z_{K} \right) \rd z_{K} \cdots \rd z_2.
		\end{align*}
		By iterated substitutions, the proof is complete for the case $N=K+1$. 
	\end{itemize}	
\end{proof}

Now we decompose the joint distribution and survival functions of the maxima of a continuous Markov process in consecutive intervals into three components: 
\begin{itemize}
	\item[a)] The distribution or survival function of the maximum of the continuous Markov process in a given interval conditional on its starting value and terminal value;
	\item[b)] The transition density function in a given interval;
	\item[c)] The distribution function of the maximum of the continuous Markov process in a given interval conditional on its starting value only.
\end{itemize}

The third item is equivalent to the FPT distribution for a Markov process to cross a constant barrier in a given interval. The first item involves the calculation of the maximum of a continuous Markov bridge, which involves Proposition \ref{prop_bridge}.

\begin{proposition}\label{prop_bridge}
	Let $(Z_t)_{t \geq 0} $ be a continuous Markov process where $Z_0 = z$, and $
	\tau_{Z,b} := \inf\left\{ t\geq 0 : Z_t \geq b \right\}$. Then,
	\begin{align*}
	\mathbb{P}\left( M_{0,T} \geq b\, | \,  Z_T=z^\prime, Z_0 = z \right) = 
	\left\{
	\begin{array}{ll}
	\int_{0}^{T}  \frac{p(t, T, b, z^\prime)}{p(0, T, z, z^\prime)} 
	f_{\tau_{Z,b}}(t;z) \rd t
	, &\text{ if } z,z^\prime < b,\\
	&\\
	1, & \text{ otherwise, }
	\end{array}
	\right.
	\end{align*}
	where $p(t, T, b, z^\prime)$ denotes the transition density function of $(Z_t)_{t \geq 0} $ from state $b$ at time $t$ to state $z^\prime$ at time T, and $f_{\tau_{Z,b}}(t;z)$ is the probability density function of the first-passage-time $\tau_{Z,b}$ with $Z_0=z$. 
\end{proposition}

\begin{proof}
See Appendix \ref{AppC}.
\end{proof}

With Proposition \ref{prop_bridge} and Theorem \ref{main_thm} at hand, we are able to at least approximate the joint distribution and survival functions of the maxima of a continuous Markov process in consecutive intervals, provided that we know its transition density function and its FPT density for a constant barrier.

\subsection{Simplified calculation of survival functions}

Now we present a theorem which simplifies the calculation of the survival function in Theorem \ref{main_thm}. We can prove that if, at the end of each interval, the terminal value of the process is lower than the barrier level in the subsequent time interval, the nested integral simplifies to a product of single integrals.
\begin{theorem}\label{sim_int}
	Given $\{ Z_{t_0} < b_1, Z_{t_1} < b_2, \ldots, Z_{t_{N-1}} < b_N \}$, the joint survival function of the maxima of a continuous Markov process $\left(Z_t\right)_{t \geq 0} \in \mathbb{R}$ in consecutive left-open and right-closed time intervals is given by
	\begin{align}\label{sim_formula}
	&\mathbb{P}\left( M_{t_0, t_1} \geq b_1, \ldots, M_{t_{N-1},t_N} \geq b_N\, | \,  {Z_{t_0} = z_0 < b_1, Z_{t_1} < b_2, \ldots, Z_{t_{N-1}} < b_N} \right)\notag\\
	&\hspace{1cm}=\mathbb{P}\left( M_{t_0, t_1} \geq b_1\, | \,  Z_{t_0} = z_0 < b_1, Z_{t_1} < b_2, \ldots, Z_{t_{N-1}} < b_N \right)\notag \\ 
	&\hspace{1.5cm}\times\prod_{i=2}^{N} \Bigg[\int_{-\infty}^{b_i} \frac{\mathbb{P}\left( M_{t_{i-1},t_i} \geq b_i\, | \,  Z_{t_{i-1}} = x_i,Z_{t_0} = z_0 < b_1, Z_{t_1} < b_2, \ldots, Z_{t_{N-1}} < b_N \right)}{\mathbb{P}\left( M_{t_{i-2},t_{i-1}} \geq b_{i-1}\, | \,  Z_{t_0} = z_0 < b_1, Z_{t_1} < b_2, \ldots, Z_{t_{N-1}} < b_N  \right)}  \notag\\
	&\hspace{1.5cm}\times\mathbb{P}\left( M_{t_{i-2},t_{i-1}} \geq b_{i-1}\, | \,  Z_{t_{i-1}} = x_i, Z_{t_0} = z_0 < b_1, Z_{t_1} < b_2, \ldots, Z_{t_{N-1}} < b_N \right)\notag\\
	&\hspace{1.5cm}\times\mathbb{P}\left( Z_{t_{i-1}} \in \rd x_i\, | \,  Z_{t_0} = z_0 < b_1, Z_{t_1} < b_2, \ldots, Z_{t_{N-1}} < b_N \right)\Bigg].
	\end{align}
\end{theorem}

\begin{proof}
	Let $\tau^{(i)} = \inf\left\{ t \geq t_{i-1}: Z_{t} = b_{i} \right\}$, 
	$\forall\ i = 1, 2, \cdots, N$. The event $\left\{ M_{t_{i-1}, t_i} \geq b_i \right\}$ is equivalent to $\left\{ \tau^{(i)} \leq t_i \right\}$; let $C=\left\{ Z_{t_1} < b_2, \ldots, Z_{t_{N-1}} < b_N \right\}$.
	Since $\left\{ \tau^{(i-1)} \leq t_{i-1} \right\} \subseteq \left\{ Z_{\tau^{(i-1)}} = b_{i-1} \right\}$, 
	$\forall\ i = 1, 2, \ldots, N$, we have 
	\begin{align*}
	& \mathbb{P}\left( M_{t_{i-1}, t_i} \geq b_i\, | \,  Z_0 = z_0, C, M_{t_0, t_1} \geq b_1, \ldots, M_{t_{i-2},t_{i-1}} \geq b_{i-1}  \right) \\
	&= \mathbb{P}\left( M_{t_{i-1}, t_i} \geq b_i\, | \,  Z_0 = z_0, C, M_{t_0, t_1} \geq b_1, \ldots, M_{t_{i-3},t_{i-2}} \geq b_{i-2},\tau^{(i-1)} \leq t_{i-1}  \right)\\
	&= \mathbb{P}\left( M_{t_{i-1}, t_i} \geq b_i\, | \,  Z_0 = z_0, C, M_{t_0, t_1} \geq b_1, \ldots, M_{t_{i-3},t_{i-2}} \geq b_{i-2},\tau^{(i-1)} \leq t_{i-1},Z_{\tau^{(i-1)}} = b_{i-1}  \right).
	\end{align*}
	Since $\left\{ M_{t_0, t_1} \geq b_1, \ldots, M_{t_{i-3},t_{i-2}} \geq b_{i-2} \right\} \subset \mathscr{F}_{\tau^{(i-1)}}$ and
	$\left\{ M_{t_{i-1}, t_i} \geq b_i \right\} \subset \mathscr{F}_{t_i}\backslash\mathscr{F}_{\tau^{(i-1)}}$,
	by Lemma \ref{markov_ind}, we obtain
	\begin{align*}
	& \mathbb{P}\left( M_{t_{i-1}, t_i} \geq b_i\, | \,  Z_0 = z_0, C, M_{t_0, t_1} \geq b_1, \ldots, M_{t_{i-2},t_{i-1}} \geq b_{i-1}  \right) \\
	&\hspace{5cm}= \mathbb{P}\left( M_{t_{i-1}, t_i} \geq b_i\, | \,  Z_0 = z_0, C, \tau^{(i-1)} \leq t_{i-1}, Z_{\tau^{(i-1)}} = b_{i-1}  \right).
	\end{align*}
	Since $\left\{ \tau^{(i-1)} \leq t_{i-1} \right\} \subseteq \left\{ Z_{\tau^{(i-1)}} = b_{i-1}  \right\}$, the above formula equals to
	\begin{align*}
	\mathbb{P}\left( M_{t_{i-1}, t_i} \geq b_i\, | \,  Z_0 = z_0, C, \tau^{(i-1)} \leq t_{i-1}  \right)=\mathbb{P}\left( M_{t_{i-1}, t_i} \geq b_i\, | \,  Z_0 = z_0, C,  M_{t_{i-2},t_{i-1}} \geq b_{i-1} \right).
	\end{align*}
	{This means that conditional on the event $C$, the discrete process $(L_i)_{i\in\mathbb{N}}$, defined by $L_i =\mathbb{I}_{ M_{t_{i-1}, t_i} \geq b_i }$, is a discrete Markov process.} 
	Hence, we have
	\begin{align*}
	&\mathbb{P}\left( M_{t_0, t_1} \geq b_1, \ldots, M_{t_{N-1},t_N} \geq b_N\, | \, Z_{t_0} = z_0 < b_1, Z_{t_1} < b_2, \ldots, Z_{t_{N-1}} < b_N \right)\\
	&= \mathbb{P}\left( M_{t_{N-1},t_N} \geq b_N\, | \,  Z_0 = z_0, C, M_{t_0, t_1} \geq b_1, \ldots, M_{t_{N-2},t_{N-1}} \geq b_{N-1} \right) \\
	&\hspace{0.75cm}\times \mathbb{P}\left( M_{t_{N-2},t_{N-1}} \geq b_{N-1}\, | \,  Z_0 = z_0, C, M_{t_0, t_1} \geq b_1, \ldots, M_{t_{N-3},t_{N-2}} \geq b_{N-2} \right)\times\ \cdots\nonumber\\
	&\hspace{0.75cm}\times\mathbb{P}\left( M_{t_0, t_1} \geq b_1\, | \,  Z_0 = z_0, C \right)\\
	&=\mathbb{P}\left( M_{t_0, t_1} \geq b_1\, | \,  Z_0 = z_0, C \right) \prod_{i=2}^{N} \mathbb{P}\left( M_{t_{i-1},t_i} \geq b_i\, | \,  Z_0 = z_0, C, M_{t_{i-2},t_{i-1}} \geq b_{i-1} \right).
	\end{align*}
	Based on the Markov property, we have
	\begin{align*}
	&\mathbb{P}\left( M_{t_{i-1},t_i} \geq b_i\, | \,  Z_0 = z_0, C, M_{t_{i-2},t_{i-1}} \geq b_{i-1} \right) \\
	&=  \int_{-\infty}^{b_i} \mathbb{P}\left( M_{t_{i-1},t_i} \geq b_i\, | \,  Z_{t_{i-1}} = x, C \right)\mathbb{P}\left( Z_{t_{i-1}} \in \rd x\, | \,  Z_0 = z_0, M_{t_{i-2},t_{i-1}} \geq b_{i-1}, C \right).
	\end{align*}
	Therefore, 
	\begin{align*}
	&\mathbb{P}\left( M_{t_0, t_1} \geq b_1, \ldots, M_{t_{N-1},t_N} \geq b_N\, | \,  Z_{t_0} = z_0 < b_1, Z_{t_1} < b_2, \ldots, Z_{t_{N-1}} < b_N \right)\\
	&= \mathbb{P}\left( M_{t_0, t_1} \geq b_1\, | \,  Z_0 = z_0,C \right)\nonumber\\
	&\times\prod_{i=2}^{N} \left[ \int_{-\infty}^{b_i} \mathbb{P}\left( M_{t_{i-1},t_i} \geq b_i\, | \,  Z_{t_{i-1}} = x, C \right)\mathbb{P}\left( Z_{t_{i-1}} \in \rd x\, | \,  Z_0 = z_0, M_{t_{i-2},t_{i-1}} \geq b_{i-1}, C \right)  \right],
	\end{align*}
	and the result stated in the theorem follows.	
\end{proof}
Here, given that $\{ Z_{t_0} < b_1, Z_{t_1} < b_2, \ldots, Z_{t_{N-1}} < b_N \}$, {the discrete process $(L_i)_{i\in\mathbb{N}}$ defined by $L_i =\mathbb{1}_{ M_{t_{i-1}, t_i} \geq b_i }$ is a discrete Markov process}. We can simplify the previous nested integral to the product of multiple single integrals under restrictions. This is significantly more efficient from a computational viewpoint. However, the joint distribution function $\mathbb{P}\left( M_{t_0, t_1} < b_1, \cdots, M_{t_{N-1},t_N} < b_N\, | \,  Z_0 = z_0 \right)$ does not admit such a simplification. One has to utilize the nested integral formula in Theorem \ref{main_thm} to compute it, although this can also be computed efficiently with the two schemes presented in Section \ref{num}. 
\subsection{Non-matching time-discretization}\label{nonmatching}
As discussed before, we may have non-matching time-discretization schemes for the piece-wise constant functions $\mu(t)$, $\lambda(t)$, $\sigma(t)$ and $b(t)$. In such a situation, the process is still continuous and Markov. By Theorem \ref{main_thm}, if we have
\begin{align*}
\mathbb{P}\left( M_{t_{i-1},t_i} \geq b_i \, | \,  Z_{t_{i-1}} = z_{i-1}, Z_{t_i} = z_i \right) \text{ and } \mathbb{P}\left( M_{t_{N-1},t_N} \geq b_N \, | \,  Z_{t_{N-1}} = z_{N-1} \right)
\end{align*}
$\forall i = 1, \ldots, N-1$, the joint distribution and survival function for the maxima of the inhomogeneous OU-process in consecutive intervals can still be calculated. We have the following two sub-cases for non-matching time-discretizations in the interval $[t_{i-1},t_i]$. 

\begin{itemize}
	\item[]
	\textbf{Case 1: Matching time-discretization for $\mu(t)$, $\lambda(t)$ and $\sigma(t)$, but non-matching for $b(t)$.} An example of this case is shown in Figure \ref{case1time}. Here, the time-discretizations for $\mu(t)$, $\lambda(t)$ and $\sigma(t)$ are the same.
	\begin{figure}[H]
		\centering
		\includegraphics[scale=0.35]{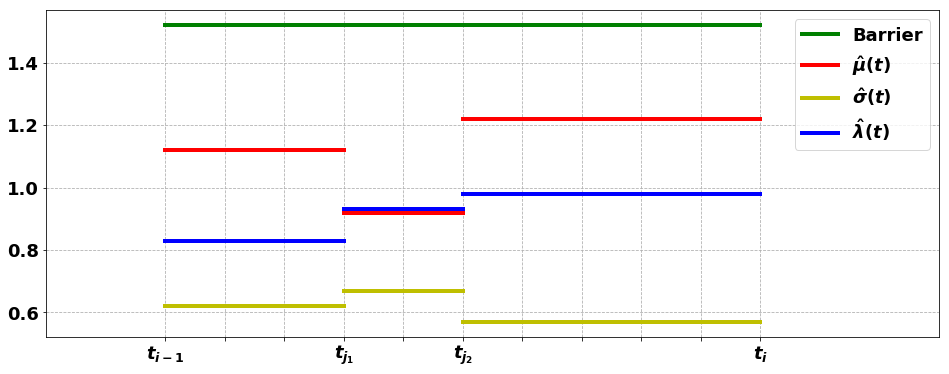}
		\caption{Matching time-discretisation for $\mu(t)$, $\lambda(t)$ and $\sigma(t)$, but  different for $b(t)$.}\label{case1time}
	\end{figure}
	In this case, $\mathbb{P}\left( M_{t_{i-1},t_i} \geq b_i \, | \,  Z_{t_{i-1}} = z_{i-1} \right)$ can still be calculated by Theorem \ref{main_thm}. We have:
	\begin{equation*}
	\mathbb{P}\left( M_{t_{i-1},t_i} \geq b_i \, | \,  Z_{t_{i-1}} = z_{i-1} \right)=1 -  \mathbb{P}\left( M_{t_{i-1},t_{j_1}} < b_i, M_{t_{j_1},t_{j_2}} < b_i, M_{t_{j_2},t_{i}} < b_i \, | \,  Z_{t_{i-1}} = z_{i-1}\right),
	\end{equation*}
	which can be solved by a nested integration formula, see Theorem \ref{main_thm}, with the local homogeneous property for each sub-interval. In terms of $\mathbb{P}\left( M_{t_{i-1},t_i} \geq b_i \, | \,  Z_{t_{i-1}} = z_{i-1}, Z_{t_i} = z_i \right)$, it also follows by Theorem \ref{main_thm} that 
	\begin{align*}
	&\mathbb{P}\left( M_{t_{i-1},t_i} \geq b_i \, | \,  Z_{t_{i-1}} = z_{i-1}, Z_{t_i} = z_i \right)\\
	&= 1 - \mathbb{P}\left( M_{t_{i-1},t_{j_1}} < b_i, M_{t_{j_1},t_{j_2}} < b_i, M_{t_{j_2},t_{i}} < b_i \, | \,  Z_{t_{i-1}} = z_{i-1}, Z_{t_{i}} = z_{i}\right)\\
	&=1- \int_{\mathbb{R}} \mathbb{P}\left( Z_{t_{j_1}}\in \rd x \, | \,  Z_{t_{i-1}} = z_{i-1}, Z_{t_{i}} = z_{i} \right) \times \\
	&\hspace{0.5cm} \mathbb{P}\left( M_{t_{i-1},t_{j_1}} < b_i, M_{t_{j_1},t_{j_2}} < b_i, M_{t_{j_2},t_{i}} < b_i \, | \,  Z_{t_{i-1}} = z_{i-1}, Z_{t_{j_1}} = x, Z_{t_{i}} = z_{i}\right)\\
	&= 1- \int_{\mathbb{R}} \mathbb{P}\left( M_{t_{i-1},t_{j_1}} < b_i \, | \,  Z_{t_{i-1}} = z_{i-1}, Z_{t_{j_1}} = x \right) \\
	&\hspace{1cm}\times \mathbb{P}\left( M_{t_{j_1},t_{j_2}} < b_i, M_{t_{j_2},t_{i}} < b_i \, | \,  Z_{t_{j_1}} = x, Z_{t_{i}} = z_{i}\right)\frac{p(t_{{j_1}}, t_{i}, x, z_{i}) p(t_{i-1}, t_{{j_1}}, z_{i-1},x)}{p(t_{i-1}, t_{i}, z_{i-1},z_i)} \rd x.
	\end{align*}
This can be simplified further to obtain	
	\begin{align*}
	&\mathbb{P}\left( M_{t_{i-1},t_i} \geq b_i \, | \,  Z_{t_{i-1}} = z_{i-1}, Z_{t_i} = z_i \right)\\
	&\quad= 1- \int_{\mathbb{R}} \mathbb{P}\left( M_{t_{i-1},t_{j_1}} < b_i \, | \,  Z_{t_{i-1}} = z_{i-1}, Z_{t_{j_1}} = x \right)\\
	&\quad\hspace{1cm}\times\int_{\mathbb{R}} \mathbb{P}\left( M_{t_{j_1},t_{j_2}} < b_i \, | \,  Z_{t_{j_1}} = x, Z_{t_{j_2}} = y\right)\mathbb{P}\left( M_{t_{j_2},t_{i}} < b_i \, | \,  Z_{t_{j_2}} = y, Z_{t_{i}} = z_{i}\right)\\
	&\hspace{1.75cm}\times \frac{p(t_{{j_2}}, t_{i}, y, z_{i}) p(t_{j_1}, t_{{j_2}}, x,y) p(t_{i-1}, t_{{j_1}}, z_{i-1},x)}{p(t_{i-1}, t_{i}, z_{i-1},z_i)}  \rd y \rd x.
	\end{align*}
	Theorem \ref{sim_int} simplifies the nested integral in Theorem \ref{main_thm} to a product of single integrals, provided that some additional constraints are satisfied. For a non-matching time-discretization, Theorem \ref{sim_int} can still be applied. However, the terms $\mathbb{P}\left( M_{t_{i-2},t_{i-1}} \geq b_{i-1} \, | \,  Z_{t_{i-1}} = x_i, Z_0 = z_0 \right)$ and $\mathbb{P}\left( M_{t_{i-1},t_i} \geq b_i \, | \,  Z_{t_{i-1}} = x_i \right)$ can only be evaluated by the nested integral in Theorem \ref{main_thm}. 
	\item[] 
	\
	
	\textbf{Case 2: Non-matching time-discretizations for any of the functions $\mu(t)$, $\lambda(t)$, $\sigma(t)$ and $b(t)$.}\\
	An example of this case is shown in Figure \ref{case2time}. This case can be reduced back to Case 1 by taking the union of all the time-discretizations steps as the overall discretizations scheme. For example in Figure \ref{case2time}, we can consider it as a special case of Case 1 for the time steps $t_{i-1}$, $t_{j_1}$, $t_{j_2}$, $t_{j_3}$, $t_{j_4}$, $t_{j_5}$ and $t_i$.
	\begin{figure}[H]
		\centering
		\includegraphics[scale=0.35]{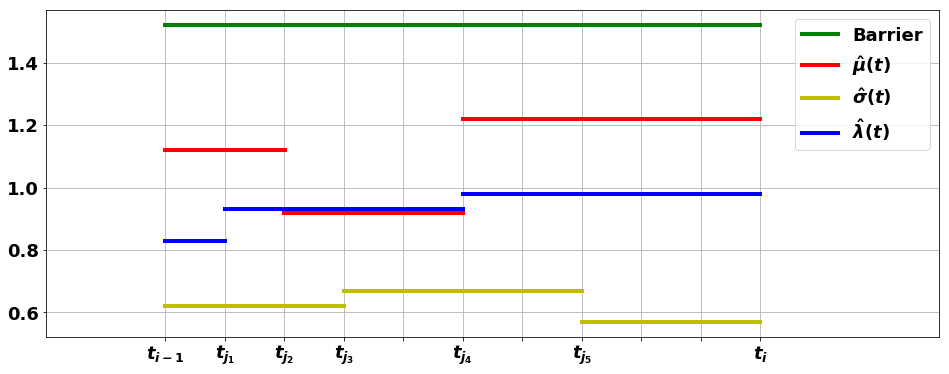}
		\caption{Non-matching time-discretization for any of $\mu(t),\lambda(t),\sigma(t)$ and $b(t)$.}\label{case2time}
	\end{figure}
\end{itemize}
\begin{remark}
	When the discretization is non-matching, we have two layers of nested integration:
	\begin{itemize}
		\item[a)] The nested integration due to non-matching time-discretization;
		\item[b)] The nested integration arising from the application of Theorem \ref{main_thm}.
	\end{itemize}
	By Theorem \ref{sim_int}, one can simplify the nested integral in b) to a product of single integrations under some restrictions. This can reduce the computational complexity. Although the nested integral in a) cannot be further reduced, in practice, if the
	variations of the piece-wise constants within a single segment are much smaller than the variations of the piece-wise constants among all the segments, one can use the matching time-discretization as an efficient approximation.
\end{remark}
\section{Computational methods and numerical results}\label{num}
Thus far, we have obtained decomposition formulae for both, the joint distribution and the survival function for the maxima of a continuous Markov process in consecutive intervals. For convenience, in this section we take the survival function for a standardized OU-process in consecutive intervals as an example to illustrate the computational methods. Moreover, for simplicity, we consider the case that the lengths of all the time intervals are constant $\Delta t$, i.e. $t_i = i \Delta t$ for $i = 0, 1, 2, ..., N$. 
\begin{corollary}\label{comp_prop}
	Let $(Z_t)_{t\geq 0}$ be a standardized OU-process and 
	$M_{t_{i-1}, t_i} = \sup_{t\in[t_{i-1}, t_i)} Z_t,$
	where $t_i = i \Delta t$, for $i = 0, 1, 2, ..., N$. Then,
	\begin{align}
	&\mathbb{P}\left( M_{t_{0},t_1} \geq b_1, \cdots, M_{t_{N-1},t_N} \geq b_N \, | \,  Z_{t_0} = z_0 \right)\notag\\ 
	&= \int_{\mathbb{R}} \kappa\left(z_{0},  z_{1}, b_1 \right) \cdots \int_{\mathbb{R}} \kappa\left( z_{N-3},  z_{N-2}, b_{N-2} \right)\int_{\mathbb{R}} \kappa\left( z_{N-2},  z_{N-1}, b_{N-1} \right) \label{comp_formula}\bar{q}\left( z_{N-1}, b_N \right) \rd z_{N-1} \rd z_{N-2} \cdots \rd z_1,
	\end{align}
	where 
	\begin{align*}
		\bar{q}\left( z_{N-1}, b_N \right) =
		\left[1 - \sum\limits_{k = 1}^{\infty} c_k^{(N)} \e^{-\alpha_k^{(N)} t} \mathscr{H}_{\alpha_k^{(N)}}\left( -z_{N-1} \right) \right] \mathbb{1}\left({z_{N-1} < b_N}\right)+ \mathbb{1}\left({z_{N-1} \geq b_N}\right)
	\end{align*}	
and		
	\begin{align*}	
		\kappa(z_{i-1}, z_i, b_i) =& \sum_{k=1}^{\infty}  {c_k^{(i)} \alpha_k^{(i)} \mathscr{H}_{\alpha_k^{(i)}}( -z_{N-1}) }\int_{\e^{-\Delta t}}^{1} \frac{x^{-\alpha_k^{(i)} - 1}\mathbb{1}\left({z_{i-1} < b_i}\right) \mathbb{1}\left({z_i < b_i}\right)}{\sqrt{\pi(1-x^2)} }\\
		&\hspace{4.25cm}\times\exp\left\{ -\frac{(z_i - b_i x)^2}{1-x^2}- \alpha_k^{(i)} \Delta t \right\}  {\rd x} \\
		&+ p(0, \Delta t, z_{i-1}, z_{i}) \left( 1 - \mathbb{1}\left({z_{i-1} < b_i}\right) \mathbb{1}\left({z_i < b_i}\right) \right).
	\end{align*}
	Here, $\mathscr{H}_{\alpha}(\cdot)$ is the Hermite function with parameter $\alpha$, $\left\{\alpha_k^{(i)}\right\}$ are the solutions to the equation
	$
	\mathscr{H}_{\alpha}\left( -b_i \right) = 0
	$
	and 
	$
	c_k^{(i)} = -1/(\alpha_k^{(i)} \partial_{\alpha_k^{(i)}}\mathscr{H}_{\alpha_k^{(i)}}( -b_i)). 
	$
\end{corollary}
\begin{proof}
See Appendix \ref{AppC}.
\end{proof}

The iterated integral can be approximated efficiently by quadrature schemes or Monte Carlo integration methods. We describe the two methods in what follows. 

\subsection{Quadrature scheme}
We first present a quadrature scheme to evaluate
\begin{align}
I :&=\mathbb{P}\left( M_{t_{0},t_1} \geq b_1, \cdots, M_{t_{N-1},t_N} \geq b_N \, | \, \, Z_{t_0} = z_0 \right)\notag\\
&=\int_{\mathbb{R}} \kappa\left(z_{0},  z_{1}, b_1 \right) \cdots \int_{\mathbb{R}} \kappa\left( z_{N-3},  z_{N-2}, b_{N-2} \right)\int_{\mathbb{R}} \kappa\left( z_{N-2},  z_{N-1}, b_{N-1} \right) \label{approx_int}\bar{q}\left( z_{N-1}, b_N \right) \rd z_{N-1} \cdots \rd z_1.
\end{align}
Since the OU-process is defined on $\R$, we choose a sufficiently large number $Z_{\text{max}}$ and a sufficiently small number $Z_{\text{min}}$.  We partition the domain $[Z_{\text{min}}, Z_{\text{max}}]$ into $L$ pieces of equal length $\delta z$, where the grid points are denoted $Z_{\text{min}} = z^{(1)} < z^{(2)} < \cdots < z^{(L)} = Z_{\text{max}} $. We can then approximate the integration as follows: 
\begin{proposition}\label{Quad}
	The nested integral in Equation (\ref{comp_formula}) can be approximated by the product of matrices 
	$
	I \approx \prod_{i=1}^{N-1}K_i  \bar Q (\delta z)^{N-1},
	$
	where for $i=1$
	\begin{eqnarray*}
		K_{1} &=&
		\left[ \kappa\left( z_0, z_1^{(1)} \right), \kappa\left( z_0, z_1^{(2)} \right), \cdots, \kappa\left( z_0, z_1^{(L)} \right) \right],
	\end{eqnarray*}
	for $i = 2,3,\cdots, N-1$,
	\begin{align*}
	K_i =& \left[
	\begin{array}{cccc}
	\kappa\left(z_{i-1}^{(1)}, z_{i}^{(1)}, b_i \right) & \kappa\left(z_{i-1}^{(1)}, z_{i}^{(2)}, b_i\right) & \cdots & \kappa\left(z_{i-1}^{(1)}, z_{i}^{(L)}, b_i\right)\\
	\kappa\left( z_{i-1}^{(2)}, z_{i}^{(1)} , b_i\right) & \kappa\left( z_{i-1}^{(2)}, z_{i}^{(2)} , b_i\right) & \cdots & \kappa\left( z_{i-1}^{(2)}, z_{i}^{(L)}, b_i \right)\\
	\vdots & \vdots &  & \vdots\\
	\kappa\left( z_{i-1}^{(L)}, z_{i}^{(1)}, b_i \right) & \kappa\left( z_{i-1}^{(L)}, z_{i}^{(2)}, b_i \right) & \cdots & \kappa\left( z_{i-1}^{(L)}, z_{i}^{(L)}, b_i \right)
	\end{array}
	\right],
	\end{align*}
	and 
	\begin{eqnarray*}
		\bar Q &=& \left[
		\bar q\left( z_{N-1}^{(1)}, b_N \right), \bar q\left( z_{N-1}^{(2)}, b_N \right), \cdots, \bar q\left( z_{N-1}^{(L)}, b_N \right)
		\right]^\top,
	\end{eqnarray*}
\end{proposition}
\begin{proof}
See Appendix \ref{AppC}.
\end{proof}

\begin{corollary}\label{Quad_coro}
	Assuming that $\{ Z_{t_0} < b_1, Z_{t_1} < b_2, \ldots, Z_{t_{N-1}} < b_N \}$, then the nested integral (\ref{comp_formula}) can be reduced to the product of single integrals, which can be evaluated efficiently in vector form as follows:
	\begin{align*}
	I &\approx C \left(\prod_{i=1}^{N} V_i W_i (\delta z)\right)
	\end{align*}
	where 
	\begin{align*}
	C &= \left(\prod_{i=2}^{N-1} \frac{1}{\mathbb{P}\left( M_{t_{i-1},t_{i}} \geq b_{i} \, | \,  Z_0 = z_0  \right)}\right)\\
	V_i &= \left[ \mathbb{P}\left( M_{t_{i-1},t_i} \geq b_i \, | \,  Z_{t_{i-1}} = z^{(1)} \right), \cdots, \mathbb{P}\left( M_{t_{i-1},t_i} \geq b_i \, | \,  Z_{t_{i-1}} = z^{(L)} \right) \right]\\
	W_i &= \left[\begin{array}{c}
	\mathbb{P}\left( M_{t_{i-2},t_{i-1}} \geq b_{i-1} \, | \,  Z_{t_{i-1}} = z^{(1)}, Z_0 = z_0 \right) p(0,t_{i-1};z_0,z^{(1)})\\
	\vdots\\
	\mathbb{P}\left( M_{t_{i-2},t_{i-1}} \geq b_{i-1} \, | \,  Z_{t_{i-1}} = z^{(L)}, Z_0 = z_0 \right) p(0,t_{i-1};z_0,z^{(L)})
	\end{array}
	\right].
	\end{align*} 
\end{corollary}
\begin{proof}
	The proof can be shown by rearranging Equation (\ref{sim_formula}) in Theorem \ref{sim_int}. 
\end{proof}

\begin{remark}
	One can see from the difference between Proposition \ref{Quad} and Corollary \ref{Quad_coro} that the complexity of the matrices operation with the simplification theorem is $O(NL)$, while the complexity of the original nested integral is $O(NL^2)$. 
\end{remark}

\subsection{Monte Carlo integration method}
The integral (\ref{approx_int}) can also be evaluated efficiently by an importance sampling approximation. 
\begin{proposition}\label{MCint}
	Assume $Z_1$, $Z_2$, ..., $Z_{N-1}$ are independent and identical random variables with density function $p: \mathbb{R} \rightarrow \mathbb{R}^+$, which first-order stochastically dominate $\varphi\left( z, z^\prime, u \right) := \kappa\left( z, z^\prime, u \right)/p(z^\prime)$. Let $Z_{i-1}^{(k_i)}$ be the $k_i$-th random number in the sample generated from the random variable $Z_{i-1}$, and let $L_i$ be the sample size of the random variable $Z_{i-1}$. Then the nested integral (\ref{comp_formula}) in Proposition \ref{comp_prop} can be approximated by the  product of matrices $I \approx \big(\prod_{i=1}^{N-1}{\Omega_i}/{L_i}  \big) \bar{\mathscr{Q}}$, where for $i=1$,
	\begin{align*}
	\Omega_{1} = &
	\left[ \varphi\left( z_0, Z_1^{(1)} \right), \varphi\left( z_0, Z_1^{(2)} \right), \cdots, \varphi\left( z_0, Z_1^{(K_1)} \right) \right],
	\end{align*}
	for $i = 2,3,\cdots, N-1$,
	\begin{align*}
	\Omega_i =& \left[
	\begin{array}{llll}
	\varphi\left(Z_{i-1}^{(1)}, Z_{i}^{(1)}, b_i \right) & \varphi\left(Z_{i-1}^{(1)}, Z_{i}^{(2)}, b_i\right) & \cdots & \varphi\left(Z_{i-1}^{(1)}, Z_{i}^{(K_{i})}, b_i\right)\\
	\varphi\left( Z_{i-1}^{(2)}, Z_{i}^{(1)} , b_i\right) & \varphi\left( Z_{i-1}^{(2)}, Z_{i}^{(2)} , b_i\right) & \cdots & \varphi\left( Z_{i-1}^{(2)}, Z_{i}^{(K_{i})}, b_i \right)\\
	\cdots & \cdots & \cdots & \cdots\\
	\varphi\left( Z_{i-1}^{(K_{i-1})}, Z_{i}^{(1)}, b_i \right) & \varphi\left( Z_{i-1}^{(K_{i-1})}, Z_{i}^{(2)}, b_i \right) & \cdots & \varphi\left( Z_{i-1}^{(K_{i-1})}, Z_{i}^{(K_{i})}, b_i \right)
	\end{array}
	\right],
	\end{align*}
	and
	\begin{align*}
	\bar{\mathscr{Q}} =& \left[
	\bar q\left( Z_{N-1}^{(1)}, b_N \right), \bar q\left( Z_{N-1}^{(2)}, b_N \right), \cdots, \bar q\left( Z_{N-1}^{(K_{N-1})}, b_N \right)
	\right]^\top.
	\end{align*}
\end{proposition}
\begin{proof}
	The integral can be rewritten as
	\begin{align*}
	I = &\int_{\mathbb{R}} \frac{\kappa\left( z_{0},  z_{1}, b_1 \right)}{p(z_1)}p(z_1) \cdots \int_{\mathbb{R}} \frac{\kappa\left(  z_{N-3},  z_{N-2}, b_{N-2} \right)}{p(z_{N-2})}p(z_{N-2}) \\
	&\hspace{1cm}\times \int_{\mathbb{R}} \frac{\kappa\left(  z_{N-2},  z_{N-1}, b_{N-1} \right) \bar q\left( z_{N-1}, b_N \right)}{p(z_{N-1})} p(z_{N-1}) \rd z_{N-1} \rd z_{N-2} \cdots \rd z_1,
	\end{align*}
	and further as
	\begin{align*}
	I = &\int_{\mathbb{R}} \varphi\left( z_0, z_1, b_1 \right) p(z_1) \cdots \int_{\mathbb{R}} {\varphi\left(  z_{N-3},  z_{N-2}, b_{N-2} \right)} p(z_{N-2}) \\
	&\hspace{1cm}\times \int_{\mathbb{R}} {\varphi\left(  z_{N-2},  z_{N-1}, b_{N-1} \right) \bar q\left( z_{N-1}, b_N \right)} p(z_{N-1}) \rd z_{N-1} \rd z_{N-2} \cdots \rd z_1 .
	\end{align*}
	The proof can now be continued analogously to the one for Proposition \ref{Quad}. 
	\end{proof}
\subsection{Numerical analysis}
The two methods can be compared with the direct Monte Carlo approach, which is shown in Algorithm \ref{DMC_normal} (or Algorithm \ref{DMC}, a small-memory version). Since the direct Monte Carlo method needs to be implemented by a time-discretization, this method underestimates the passage-time probability, see Lemma \ref{DMC_underestimate} in Appendix \ref{DMC_analysis}. We observe that when the number of time steps increases, the results obtained by the direct Monte Carlo method align with the results obtained by the quadrature scheme and the Monte Carlo integration. However, direct Monte Carlo results become increasingly noisy when the joint passage event becomes rarer, while the quadrature scheme and Monte Carlo integration methods remain stable. This shows the quadrature and Monte Carlo integration methods can improve the accuracy if the joint passage is an infrequent event. Another interesting insight is that although the direct Monte Carlo result is more accurate for a larger number of time steps, its Monte Carlo error is bigger, too, provided that the event occurrence is infrequent, see Lemma \ref{DMC_error} in Appendix \ref{DMC_analysis}. We conclude that the direct Monte Carlo method is not suitable for the passage-time approximation. 

In the comparison of the three methods, i.e. direct Monte Carlo, quadrature scheme and Monte Carlo integration, we test the following two cases:
\begin{itemize}
	\item[1)] 
	We fix the number of consecutive intervals and change the level of the barriers in each interval;
	\item[2)]
	We fix the level of the barriers and increase the number of intervals. 
\end{itemize}
In both cases, the probability we wish to approximate becomes small when the barrier levels rise or the number of intervals increases. The direct Monte Carlo estimator will become noisy when the joint event becomes rare, see Lemma \ref{DMC_underestimate} and \ref{DMC_error} in Appendix \ref{DMC_analysis}. In this subsection, we show that the quadrature scheme and Monte Carlo integration estimators produce accurate and robust approximations and the efficiency is improved compared with the direct Monte Carlo method. The quadrature scheme and the Monte Carlo integration scheme contain two types of error source:
\begin{itemize}
	\item[1)] The truncation error from the approximation of the FPT density infinite series, and
	\item[2)] the deterministic or stochastic error from the numerical integration. 
\end{itemize}
As shown in Section \ref{homo}, the truncation error in 1) can be reduced to a small level by introducing few truncation terms. The numerical error 2) depends on its discretization size in the quadrature scheme and on the number of paths in the Monte Carlo integration scheme. This type of error can be reduced by introducing a finer discretization and/or by producing more Monte Carlo samples.

In the first numerical example we compute the probability of the maxima of a standardized OU-process in the first and the second periods to cross the levels $b_1$ and $b_2$, respectively. In the left column of Table \ref{Tab1}, seven combinations of barrier levels for the two consecutive periods are considered. In this example, we choose the number of paths and discrete time steps so to achieve an accuracy of the order $10^{-4}$ for the case $b_1 = 1$ and $b_2 = 1$.
{
	\begin{table}[H]
		\centering
		\begin{tabular}{cc|c|c|c|c|c}
			$b_1$ & $b_2$ & MC (500) & MC (1000) & MC (2000) & Quad. & MC int. \\
			\hline
			\hline
			1&1 &  $1.417\times 10^{-1}$ & $1.448\times 10^{-1}$ & $1.469\times 10^{-1}$ & $1.517\times 10^{-1}$ & $1.515\times 10^{-1}$\\
			&& $(2 \times 10^{-4})$ & $(2 \times 10^{-4})$ & $(2 \times 10^{-4})$ & $(3\times 10^{-4})$& $(6 \times 10^{-4})$\\
			\hline
			1&2 &$1.27 \times 10^{-2}$ &$1.311 \times 10^{-2}$ &$1.352 \times 10^{-2}$ &$1.426 \times 10^{-2}$ & $1.440 \times 10^{-2}$\\
			&& $(1\times 10^{-4})$ & $(8\times 10^{-5})$ & $(8\times 10^{-5})$ &$(2\times 10^{-5})$ & $(7\times 10^{-5})$\\
			\hline
			2&1 & $5.08\times 10^{-3}$ &$5.38\times 10^{-3}$ &$5.54\times 10^{-3}$ &$5.837\times 10^{-3}$ &$5.843\times 10^{-3}$\\
			&& $(7\times 10^{-5})$ & $(5\times 10^{-5})$ & $(6\times 10^{-5})$ & $(2\times 10^{-5})$& $(3\times 10^{-5})$\\
			\hline
			2&2 &$2.35\times 10^{-3}$ &$2.50\times 10^{-3}$ &$2.62\times 10^{-3}$ &$2.72\times 10^{-3}$ &$2.74\times 10^{-3}$\\
			&& $(4\times 10^{-5})$ & $(4\times 10^{-5})$ & $(4\times 10^{-5})$ & $(2\times 10^{-5})$& $(3\times 10^{-5})$\\
			\hline
			2&3 &$4.1\times 10^{-5}$ &$5.0\times 10^{-5}$ &$5.4\times 10^{-5}$ &$5.08\times 10^{-5}$ &$5.10\times 10^{-5}$\\
			&& $(4\times 10^{-6})$ & $(5\times 10^{-6})$ & $(4\times 10^{-6})$ &$(2\times 10^{-7})$ & $(3\times 10^{-7})$ \\
			\hline
			3&2 &$1.1\times 10^{-5}$ &$1.1\times 10^{-5}$ &$1.2\times 10^{-5}$ &$1.455\times 10^{-5}$ &$1.458\times 10^{-5}$\\
			&& $(3\times10^{-6})$ & $(3\times10^{-6})$ & $(4\times10^{-6})$ & $(1\times 10^{-7})$& $(9\times10^{-8})$ \\
			\hline
			3&3 &$7\times 10^{-6}$ &$6\times 10^{-6}$ &$6\times 10^{-6}$& $5.47\times 10^{-6}$ &$5.42\times 10^{-6}$\\
			&& $(2\times 10^{-6})$ & $(2\times 10^{-6})$ & $(2\times 10^{-6})$ & $(9\times 10^{-8})$& $(5\times 10^{-8})$\\
			\hline
			\hline
		\end{tabular}
		\caption{The probability of the maxima for a standardized OU-process crossing the barrier levels $b_1$ and $b_2$ in the two consecutive time intervals, respectively. The number in the bracket is the absolute Monte Carlo or quadrature error. The three sets of direct Monte Carlo results are implemented with $2,000,000$ sample paths. The number of time steps for the three sets of direct Monte Carlo results are $500$, $1,000$ and $2,000$, respectively. The quadrature scheme is implemented between the state domain $[-5, 5]$ with state increment $0.005$. The Monte Carlo integration method is implemented with $100,000$ sample paths.}\label{Tab1}
	\end{table}
}
We can also observe from Table \ref{Tab1} that in the direct Monte Carlo cases, the error cannot be improved by introducing a finer time-discretisation, see Lemma \ref{DMC_error} in Appendix \ref{DMC_analysis}. We also compare the computational times needed to obtain the results in Table \ref{Tab1}. In Figure \ref{Fig1}, we observe that the quadrature and Monte Carlo integration methods are more efficient than the direct Monte Carlo scheme. It turns out that the quadrature scheme performs best. 

\begin{figure}[H]
	\centering
	\includegraphics[scale=0.2]{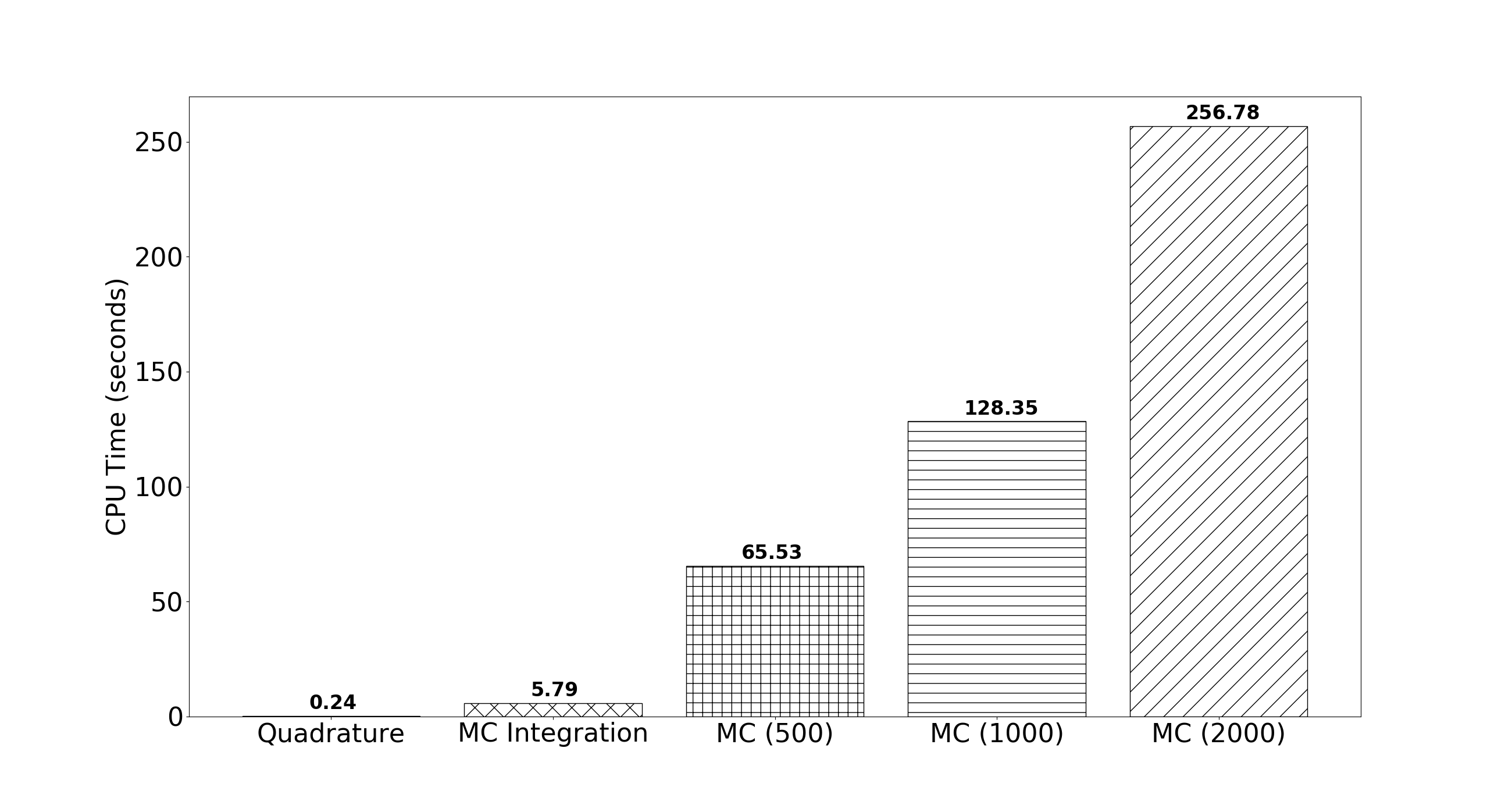}
	\caption{CPU time consumption for the case that $b_1 = b_2=2$, with prescribed maximum absolute error of $2\sim 4 \times 10^{-5}$ for all schemes for the computation of the multiple crossing probability in the two consecutive time periods. }\label{Fig1}
\end{figure}

Our second example is to fix the barrier and increase the number of intervals, which leads to lower joint passage probabilities. In Table \ref{Tab2}, we observe that the numerical results obtained by the direct Monte Carlo methods tend to be less accurate when the number of intervals increases. However, the quadrature and Monte Carlo integration methods remain reliable compared with the direct Monte Carlo results. 

\begin{table}[H]
	\centering
	\begin{tabular}{c|c|c|c|c|c}
		N & MC (500) & MC (1000) & MC (2000) & Quadrature & MC integration \\
		\hline
		\hline
		2 &$2.35\times 10^{-3}$ &$2.50\times 10^{-3}$&$2.62\times 10^{-3}$ &$2.72\times 10^{-3}$ &$2.74\times 10^{-3}$\\
		& $(5\times 10^{-5})$ & $(4\times 10^{-5})$ & $(4\times 10^{-5})$ & $(2\times 10^{-5})$& $(3\times 10^{-5})$\\
		\hline
		3 &$3.1\times 10^{-4}$ &$3.3\times 10^{-4}$ &$3.3\times 10^{-4}$ & $3.23\times 10^{-4}$ & $3.29\times 10^{-4}$ \\
		& $(1\times 10^{-5})$ & $(1\times 10^{-5})$ & $(1\times 10^{-5})$ &$(1\times 10^{-6})$ & $(5\times 10^{-6})$\\
		\hline
		4 &$5.5\times 10^{-5}$ &$5.7\times 10^{-5}$ &$5.6\times 10^{-5}$ & $5.23\times 10^{-5}$ & $5.30\times 10^{-5}$ \\
		& $(4\times 10^{-6})$ & $(7\times 10^{-6})$ & $(8\times 10^{-6})$ &$(3\times 10^{-7})$ & $(8\times 10^{-7})$\\
		\hline
		5 &$8\times 10^{-6}$&$7\times 10^{-6}$ &$8\times 10^{-6}$ & $8.06\times 10^{-6}$ & $8.07\times 10^{-6}$ \\
		& $(2\times 10^{-6})$ & $(2\times 10^{-6})$ & $(2\times 10^{-6})$ & $(6\times 10^{-8})$& $(9\times 10^{-8})$\\
		\hline
		\hline
	\end{tabular}
	\caption{Probability of the maxima of a standardized OU-process to be above the barrier level $b = 2$ in $N$ consecutive intervals. The number below is the absolute Monte Carlo or quadrature error. The three sets of direct Monte Carlo results are implemented with $2,000,000$ sample paths. The number of time steps for the three sets of direct Monte Carlo results are $500$, $1,000$ and $2,000$, respectively. The quadrature scheme is implemented between the state domain $[-5, 5]$ with state increment $0.005$. The Monte Carlo integration method is implemented with $100,000$ sample paths.}\label{Tab2}
\end{table}
We deduce from Figure \ref{Fig2} that the time needed when using the direct Monte Carlo method increases linearly with respect to the number of consecutive intervals considered. On the other hand, there is a small jump in the time consumption for the quadrature scheme and the Monte Carlo integration method. This is because when only two intervals are considered, the matrix $K_i$ or $\Omega_i$ in Proposition \ref{Quad} and \ref{MCint} is not necessary. Once the matrix $K_i$ or $\Omega_i$ is obtained, it is saved for further computations. This shows that the computational demand of the quadrature scheme and the Monte Carlo integration method remain essentially unchanged when considering three or more consecutive intervals. The probability in the case of a large number of intervals can be evaluated more efficiently by the quadrature and Monte Carlo integration methods. 
\begin{figure}[H]
	\begin{center}
	\includegraphics[scale=0.2]{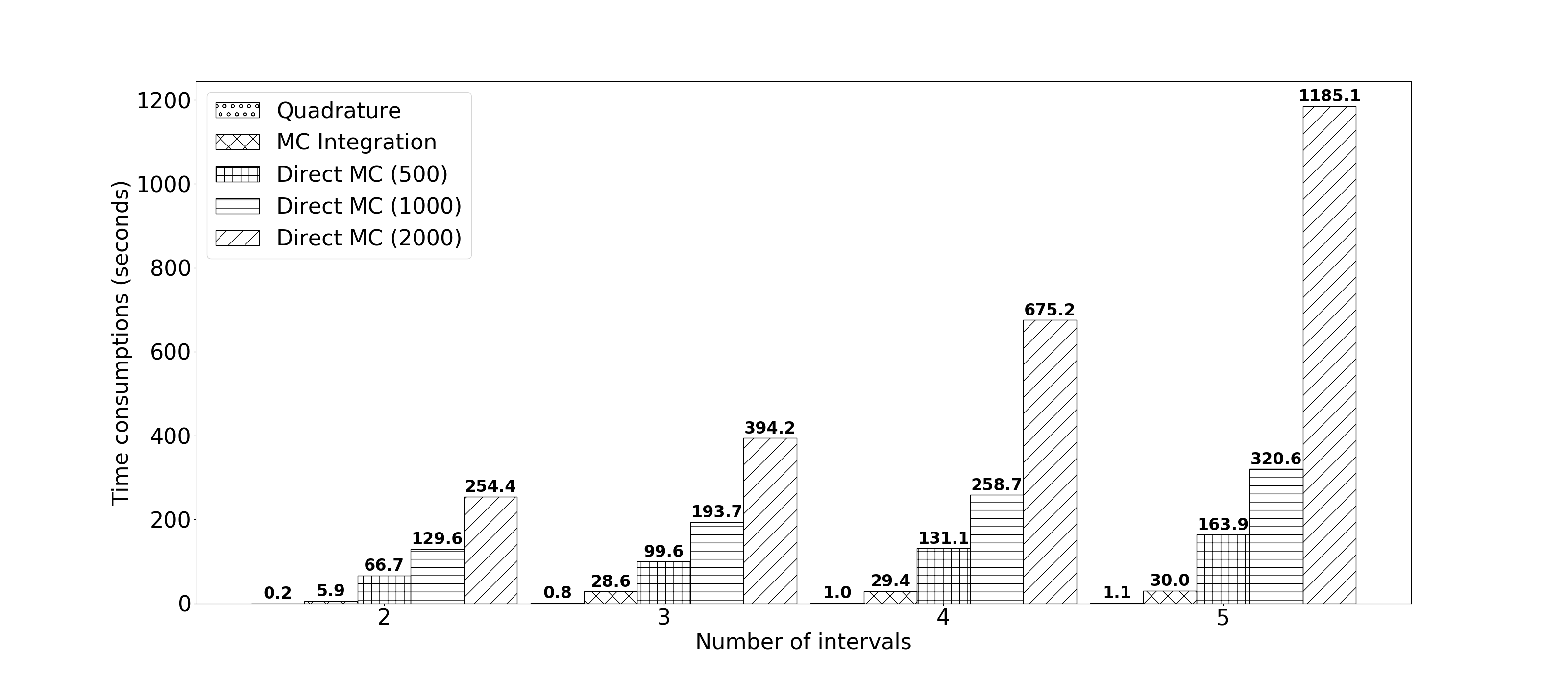}
	\caption{Time consumption of different schemes for an increasing number of intervals considered in Table \ref{Tab2}. }\label{Fig2}
	\end{center}
\end{figure}
\section{Conclusions}
We consider the multiple barrier-crossing problem of an Ornstein-Uhlenbeck (OU) process in consecutive periods of time and  focus on deriving the joint distribution and survival functions of the maxima in fixed---though arbitrary---subsequent time periods. To our knowledge, this is the first time this
mathematical problem has been formulated and tackled, while developing the needed mathematical theory, at the same time. We next summarise the main theoretical and methodological results, along with the outcomes and insights of our analysis obtained on the way:
\begin{itemize}
	\item[a)] We provide a generalization for the known infinite-series representation of the first-passage time (FPT) of a homogeneous OU-process. The extension is obtained by moving the lower-reflection boundary to minus infinity, which relaxes a condition thus far used in the literature to derive the analytical expression for the FPT. In doing so, we also provide an alternative proof for the FPT result.
	
	\item[b)] We produce the analysis of the truncation error of the generalized infinite-series representation for the OU-FPT. One is thus in the position to decide how many terms of the infinite series are necessary in order to achieve a required precision level of the first-passage probability.
	
	\item[c)] We study the tail-behaviour of the FPT survival function and conclude that it is light-tailed.
	
	\item[d)] We provide, what one might term, the {\it FPT transformation (or equivalence) theorem}: the problem of a time-inhomogeneous OU-process crossing a time-varying barrier is transformed to the equivalent problem of a homogeneous (or standardized) OU process crossing a different time-varying barrier.
	
	\item[e)] For the problem of an inhomogeneous OU-process crossing a time-varying barrier, we produce a detailed analysis that compares the errors between (i) a piece-wise constant approximation applied directly to the time-dependent parameters of the inhomogeneous OU-process and the time-varying barrier and (ii) a piece-wise constant approximation of the time-varying barrier after the ``inhomogeneous FPT-problem'' is transformed to the ``standardized FPT-problem''. We provide a criterion useful for choosing between the two schemes.
	\item[f)] Given the FPT distribution function and the transition density function of a Markov process, we obtain a semi-analytical formula for the joint distribution and survival functions of the maxima of a continuous Markov process in consecutive time periods. 
	
	\item[g)] Given the transition density function and the FPT density function of a Markov process, we provide a semi-analytical formula for the FPT distribution function of a Markov bridge process.		
	
	\item[h)] By adding a (mild) condition, we provide a simplification theorem that reduces the nested integration, appearing in the semi-analytical formula for the joint distribution and survival functions of the maxima of a continuous Markov process in consecutive time periods, to a product of single integrals. 
	
	\item[i)] We provide two efficient and robust computational methods to compute the nested integration emerging in the semi-analytical formula for the joint distribution and survival functions of the maxima in consecutive time periods. The numerical results confirm efficiency and accuracy of the quadrature method, in particular. 
	
	\item[j)] We show that Monte Carlo methods, applied for the computation of the distribution and survival functions, a) underestimate the probability of the barrier-crossing event and b) increase the computational error when the number of sample paths is increased. This deficiencies are further exacerbated in the case that the (barrier-crossing) event is rare, just as it would be if, for example, the considered event were a heatwave.
\end{itemize}
\section*{Acknowledgements}
The authors are grateful to participants of the 6th International Conference of Mathematics in Finance, Kruger National Park, South Africa (August 2017), the London-Paris Bachelier Workshop on Mathematical Finance, University College London, U. K. (September 2017), the Fourth Young Researchers Meeting on BSDEs, Nonlinear Expectations and Mathematical Finance, Shanghai Jiaotong University, Shanghai, China (April 2018), the EMAp Research Seminar, Funda\c c\~ao Getulio Vargas, Rio de Janeiro (August 2018), and of the Seminar of the Department of Mathematical Sciences, University of Copenhagen (May 2019) for comments and suggestions. The authors thank Prof. Tomoko Matsui and the Institute of Statistical Mathematics in Tokyo as well as Prof. David Taylor and the African Institute for Financial Markets \& Risk Management (AIFMRM), University of Cape Town, for facilitating aspects of this research through presentations and research visits. The authors are thankful for the suggestions for improvements provided by anonymous reviewers. 

\section*{Appendix}
\appendix
\section{Algorithms}
\begin{algorithm}[H]
	\begin{algorithmic}[1]
		\State For a given initial value $x$ and barrier level $b$, compute all the $\alpha$-zeros in a given interval. 
		\State Take all $\alpha$-zeros to approximate the $\omega$ quantile of the hitting time, denoted $q_{\omega}$, by Theorem \ref{OUhit}. 
		\State Fix the relative error $\delta$. Obtain the error tolerance interval $[q_{\omega}(1-\delta), q_{\omega}(1+\delta)]$.
		\State Denote the approximation with $n$ ordered $\alpha$-zeros by $\hat{q}_{\omega}$. 
		\State Starting from $n=1$: 
		\While {$\hat{q}_{\omega} \notin [q_{\omega}(1-\delta), q_{\omega}(1+\delta)]$}
		\State $n += 1$.
		\EndWhile
		\State Output $n$. 
	\end{algorithmic}
	\caption{Truncation terms deviation by quantile.}\label{algor_color}
\end{algorithm}
Notation:
\begin{table}[H]
\begin{tabular}{lll}
&$NB$: &Number of sets \\
&$N$: &Number of paths for a given simulation set\\
&$M$: &Number of time steps per path\\
& $b_1$, $b_2$: &barrier level in consecutive intervals one and two\\
& $n$: &$n$-th path\\
& $m$: &$m$-th time step\\
& $\phi_m^{(n)}$:& Realisation of standard normal random variable for path 				$n$ at time step $m$\\
& $\delta = T / M$: &Length of time step\\
& $x_{m}^{(n)}$: &Realised OU-process value of path $n$ at time step $m$\\
& $\mathbb{I}^{(n)}$: &Indicator function\\
& $\textit{Prob}_{(\textit{nb})}$: &Probability of joint crossing in two consecutive intervals for set $nb$
\end{tabular}
\end{table}
\begin{algorithm}[H]
	\caption{Algorithm of Direct Monte Carlo}\label{DMC_normal}
	\begin{algorithmic}[1]
		\While {path $\textit{n}\leq \textit{N}$}
		\While {time step $\textit{m}\leq \textit{M}$}
		\State 1, simulate realization of standard normal random variable $\phi_m^{(n)}$
		\State 2, evaluate $x_{m+1}^{(n)} = x_{m}^{(n)} \e^{-\lambda \delta } + \mu (1 - \e^{-\lambda \delta }) + \sigma\sqrt{\frac{1 - \e^{-2\lambda \delta }}{2\lambda}} \phi_m^{(n)}$
		\EndWhile
		
		\State \If {$\max( x_{1}^{(n)}, \cdots, x_{M/2}^{(n)} ) \geq b_1$ \textbf{and} $\max( x_{M/2+1}^{(n)}, \cdots, x_{M}^{(n)} ) \geq b_2$}
		\State $\mathbb{I}^{(n)} = 1$
		\Else
		\State $\mathbb{I}^{(n)} = 0$
		\EndIf
		\EndWhile
		\State {$\textit{Final\_Prob} = \text{Mean}(\textit{Ind})$}
		\State {$\textit{Final\_Err} = \text{StD}(\textit{Ind}) / \sqrt{\textit{N}}$}
	\end{algorithmic}
\end{algorithm}
\begin{algorithm}[H]
	\caption{Algorithm of Direct Monte Carlo (low memory requirement)}\label{DMC}
	\begin{algorithmic}[1]
		\While {set $\textit{nb}\leq \textit{NB}$}
		\While {path $\textit{n}\leq \textit{N}$}
		\While {time step $\textit{m}\leq \textit{M}$}
		\State 1, simulate realization of standard normal random variable $\phi_m^{(n)}$
		\State 2, evaluate $x_{m+1}^{(n)} = x_{m}^{(n)} \e^{-\lambda \delta } + \mu (1 - \e^{-\lambda \delta }) + \sigma\sqrt{\frac{1 - \e^{-2\lambda \delta }}{2\lambda}} \phi_m^{(n)}$
		\EndWhile
		
		\State \If {$\max( x_{1}^{(n)}, \cdots, x_{M/2}^{(n)} ) \geq b_1$ \textbf{and} $\max( x_{M/2+1}^{(n)}, \cdots, x_{M}^{(n)} ) \geq b_2$}
		\State $\mathbb{I}^{(n)} = 1$
		\Else
		\State $\mathbb{I}^{(n)} = 0$
		\EndIf
		\EndWhile
		\State {$\textit{Prob}_{(\textit{nb})} = \text{Mean}({\textit{Ind}})$}
		\EndWhile
		\State {$\textit{Final\_Prob} = \text{Mean}(\textit{Prob})$}
		\State {$\textit{Final\_Err} = \text{StD}(\textit{Prob}) / \sqrt{\textit{NB}}$}
	\end{algorithmic}
\end{algorithm}
\section{Error analysis of the direct Monte Carlo method for FPT estimation}\label{DMC_analysis}
We estimate the probability 
$\mathbb{P}\left( \sup_{t\in(t_{0}, t_1]} X_t \geq b_1, \sup_{t\in(t_{1}, t_2]} X_t \geq b_2 \right)$
by the following algorithm:
\begin{itemize}
	\item[a)]
	We discretize the time interval $[t_0, t_2]$ into $M$ pieces: 
	$$t_0=t^{(0)} < t^{(1)} < \cdots < t^{(\frac{M}{2})} < t_1 = t^{(\frac{M}{2}+1)} < \cdots < t^{(M)} = t_2 . $$ 
	\item[b)]
	We estimate the maximum in each interval by
	\begin{align*}
	&\sup_{t\in(t_{0}, t_1]} X_t \approx \max\left\{ X_{t^{(0)}}, X_{t^{(1)}}, \cdots, X_{t^{(\frac{M}{2})}} \right\},&\\
	&\sup_{t\in(t_{1}, t_2]} X_t \approx \max\left\{ X_{t^{(\frac{M}{2}+1)}}, X_{t^{(\frac{M}{2}+2)}}, \cdots, X_{t^{({M})}} \right\}.&
	\end{align*}
	\item[c)]
	We approximate $\mathbb{P}\left( \sup_{t\in(t_{0}, t_1]} X_t \geq b_1, \sup_{t\in(t_{1}, t_2]} X_t \geq b_2 \right)$ with
	\begin{align*}
	&\mathbb{P}\left( \max\left\{ X_{t^{(1)}}, X_{t^{(2)}}, \cdots, X_{t^{(\frac{M}{2})}} \right\} \geq b_1, \max\left\{ X_{t^{(\frac{M}{2}+1)}}, X_{t^{(\frac{M}{2}+2)}}, \cdots, X_{t^{({M})}} \right\} \geq b_2 \right). 
	\end{align*}
\end{itemize}

\begin{lemma}\label{DMC_underestimate}
	The direct Monte Carlo algorithm (a)-(c) underestimates the actual probability due to the time-discretization, that is
	\begin{align*}
	& \mathbb{P}\left( \sup_{t\in(t_{0}, t_1]} X_t \geq b_1, \sup_{t\in(t_{1}, t_2]} X_t \geq b_2 \right) \\
	&\geq  \mathbb{P}\left( \max\left\{ X_{t^{(1)}}, X_{t^{(2)}}, \cdots, X_{t^{(\frac{M}{2})}} \right\} \geq b_1, \max\left\{ X_{t^{(\frac{M}{2}+1)}}, X_{t^{(\frac{M}{2}+2)}}, \cdots, X_{t^{({M})}} \right\} \geq b_2 \right). 
	\end{align*}
\end{lemma}

\begin{proof}
	
	We observe that
	\begin{align*}
	&\left\{\max\left\{ X_{t^{(1)}}, X_{t^{(2)}}, \cdots, X_{t^{(\frac{M}{2})}} \right\} \geq b_1,  \max\left\{ X_{t^{(\frac{M}{2}+1)}}, X_{t^{(\frac{M}{2}+2)}}, \cdots, X_{t^{({M})}} \right\}  \geq b_2 \right\}\\
	&\subseteq \left\{\sup_{t\in(t_{0}, t_1]} X_t \geq b_1, \sup_{t\in(t_{1}, t_2]} X_t \geq b_2 \right\}.
	\end{align*}
	Since the probability of a sub-event is smaller than that of the event itself, we have
	\begin{align*}
	& \mathbb{P}\left( \sup_{t\in(t_{0}, t_1]} X_t \geq b_1, \sup_{t\in(t_{1}, t_2]} X_t \geq b_2 \right) \\
	&\geq\mathbb{P}\left( \max\left\{ X_{t^{(1)}}, X_{t^{(2)}}, \cdots, X_{t^{(\frac{M}{2})}} \right\} \geq b_1, \max\left\{ X_{t^{(\frac{M}{2}+1)}}, X_{t^{(\frac{M}{2}+2)}}, \cdots, X_{t^{({M})}} \right\} \geq b_2 \right). 
	\end{align*}
\end{proof}
\noindent The larger $M$, the more time steps, and the approximated probability tends to be closer to the actual probability of the event (that cannot be accurately estimated by direct Monte Carlo). For example, let us consider $M = 500, 1000, 2000$ partitions in the time interval $[0,2]$. We then have:
\begin{align*}
& \mathbb{P}\big( \sup_{t\in(0, 1]} X_t \geq b_1, \sup_{t\in(1, 2]} X_t \geq b_2 \big) \\
&\geq \mathbb{P}\left( \max\left\{ X_{1/1000}, X_{2/1000}, \cdots, X_{1000/1000} \right\} \geq b_1, \max\left\{ X_{1001/1000}, X_{1002/1000}, \cdots, X_{2000/1000} \right\} \geq b_2 \right)\\
&\geq  \mathbb{P}\left( \max\left\{ X_{2/1000}, X_{4/1000}, \cdots, X_{1000/1000} \right\} \geq b_1,\max\left\{ X_{1002/1000}, X_{1004/1000}, \cdots, X_{2000/1000} \right\} \geq b_2 \right)\\
&= \mathbb{P}\left( \max\left\{ X_{1/500}, X_{2/500}, \cdots, X_{500/500} \right\} \geq b_1,\max\left\{ X_{501/500}, X_{502/500}, \cdots, X_{1000/500} \right\} \geq b_2 \right)\\
&\geq \mathbb{P}\left( \max\left\{ X_{2/500}, X_{4/500}, \cdots, X_{500/500} \right\} \geq b_1,\max\left\{ X_{502/500}, X_{504/500}, \cdots, X_{1000/500} \right\} \geq b_2 \right)\\
&= \mathbb{P}\left( \max\left\{ X_{1/250}, X_{2/250}, \cdots, X_{250/250} \right\} \geq b_1,\max\left\{ X_{251/250}, X_{252/250}, \cdots, X_{500/250} \right\} \geq b_2 \right).
\end{align*}
Therefore,
$\text{prob}_{\text{MC(500)}} \leq \text{prob}_{\text{MC(1000)}} \leq \text{prob}_{\text{MC(2000)}} \leq \text{prob}_{\text{actual}}$. We next address the Monte Carlo error of the algorithm (a)-(c).
\begin{lemma}\label{DMC_error}
	For a fixed number of paths, the errors of the direct Monte Carlo algorithm (a)-(c), based on a discretization with $M_1$ and $M_2$ time steps ($M_1 < M_2$), satisfy the relations
	\begin{align*}
	\textit{Err}_{\text{MC}(M_1)} \geq  \textit{Err}_{\text{MC}(M_2)} \quad \text{ if and only if } \quad \text{prob}_{\text{MC}(M_1)} + \text{prob}_{\text{MC}(M_2)}\geq  1,\\
	\textit{Err}_{\text{MC}(M_1)} \leq  \textit{Err}_{\text{MC}(M_2)} \quad \text{ if and only if } \quad \text{prob}_{\text{MC}(M_1)} + \text{prob}_{\text{MC}(M_2)}\leq  1. 
	\end{align*}
\end{lemma}
\begin{proof}
	For convenience, we write
	\begin{align*}
	&A:=\left\{\max\left\{ X_{t^{(1)}}, X_{t^{(2)}}, \cdots, X_{t^{(\frac{M_1}{2})}} \right\} \geq b_1,  \max\left\{ X_{t^{(\frac{M_1}{2}+1)}}, X_{t^{(\frac{M_2}{2}+2)}}, \cdots, X_{t^{({M_1})}} \right\}  \geq b_2\right\},\\
	&B:=\left\{\max\left\{ X_{t^{(1)}}, X_{t^{(2)}}, \cdots, X_{t^{(\frac{M_2}{2})}} \right\} \geq b_1,  \max\left\{ X_{t^{(\frac{M_2}{2}+1)}}, X_{t^{(\frac{M_2}{2}+2)}}, \cdots, X_{t^{({M_2})}} \right\}  \geq b_2\right\}.
	\end{align*}
	Then the Monte Carlo algorithm (a)-(c) computes
	$
	\text{prob}_{\text{MC}(M_1)} = \mathbb{P}(A) = \mathbb{E}[ \mathbb{1}_{A}]
	$
	for $M_1$ time steps, and 
	$
	\text{prob}_{\text{MC}(M_2)} = \mathbb{P}(B) = \mathbb{E}[ \mathbb{1}_{B}]
	$
	for $M_2$ time steps. If we implement the Monte Carlo algorithm (a)-(c) to compute $\mathbb{E}\left[ \mathbb{1}_{A} \right]$ and $\mathbb{E}\left[ \mathbb{1}_{B} \right]$, the ratio between the resulting errors is equal to the ratio between the standard deviations of $\mathbb{1}_{A}$ and $\mathbb{1}_{B}$. That is:
	\begin{align*}
	\frac{\textit{Err}({\mathbb{1}_{A}})}{\textit{Err}({\mathbb{1}_{B}})} =& \frac{\textit{StD}({\mathbb{1}_{A}})}{\textit{StD}({\mathbb{1}_{B}})} = \sqrt{\frac{\textit{Var}({\mathbb{1}_{A}})}{\textit{Var}({\mathbb{1}_{B}})}} = \sqrt{\frac{ \mathbb{E}\left[ \left( \mathbb{1}_{A} \right)^2 \right] - \mathbb{E}\left[ \mathbb{1}_{A} \right]^2 }{ \mathbb{E}\left[ \left( \mathbb{1}_{B} \right)^2 \right] - \mathbb{E}\left[ \mathbb{1}_{B} \right]^2 }}.
	\end{align*}
	We observe that $\left(\mathbb{1}_{A}\right)^2 = \mathbb{1}_{A}$ and
	$\left(\mathbb{1}_{B}\right)^2 = \mathbb{1}_{B}$. Therefore, 
	\begin{align*}
	\frac{\textit{Err}({\mathbb{1}_{A}})}{\textit{Err}({\mathbb{1}_{B}})} &= \sqrt{\frac{ \mathbb{E}\left[ \left( \mathbb{1}_{A} \right)^2 \right] - \mathbb{E}\left[ \mathbb{1}_{A} \right]^2 }{ \mathbb{E}\left[ \left( \mathbb{1}_{B} \right)^2 \right] - \mathbb{E}\left[ \mathbb{1}_{B} \right]^2 }} = \sqrt{\frac{ \mathbb{E}\left[ \mathbb{1}_{A} \right] - \mathbb{E}\left[ \mathbb{1}_{A} \right]^2 }{ \mathbb{E}\left[  \mathbb{1}_{B}  \right] - \mathbb{E}\left[ \mathbb{1}_{B} \right]^2 }}= \sqrt{\frac{ \mathbb{E}\left[ \mathbb{1}_{A} \right] \left(  1 - \mathbb{E}\left[ \mathbb{1}_{A} \right]\right) }{ \mathbb{E}\left[  \mathbb{1}_{B}  \right] \left( 1 - \mathbb{E}\left[ \mathbb{1}_{B} \right] \right) }} = \sqrt{ \frac{\mathbb{P}(A) \mathbb{P}(\bar A) }{\mathbb{P}(B) \mathbb{P}(\bar B) } }.
	\end{align*}
	This shows that
	\begin{align*}
	\textit{Err}_{\text{MC}(M_1)} = \sqrt{ \frac{ \text{prob}_{\text{MC}(M_1)} }{ \text{prob}_{\text{MC}(M_2)} } \frac{ 1 - \text{prob}_{\text{MC}(M_1)} }{ 1 - \text{prob}_{\text{MC}(M_2)} } }\ \textit{Err}_{\text{MC}(M_2)}.
	\end{align*}
	
	\noindent Since $M_1 < M_2$, we have
	$\text{prob}_{\text{MC}(M_1)} \leq \text{prob}_{\text{MC}(M_2)}$.	
	We denote $\text{prob}_{\text{MC}(M_2)}$ by $p\in(0, 1)$, then $\text{prob}_{\text{MC}(M_1)} = p - a$, for some $a\in [0, p]$. Therefore, 
	\begin{align*}
	\textit{Err}_{\text{MC}(M_1)} =& \sqrt{ \frac{ p-a }{ p } \cdot \frac{ 1 - (p-a) }{ 1 - p } }\ \textit{Err}_{\text{MC}(M_2)}\\
	=& \sqrt{ 1 + \frac{a}{p(1-p)}\left( \text{prob}_{\text{MC}(M_1)} + \text{prob}_{\text{MC}(M_2)} -1 \right) }\ \textit{Err}_{\text{MC}(M_2)}.
	\end{align*}
	Since $a\geq 0$ and $p\in[0, 1]$, 	
	$$
	\textit{Err}_{\text{MC}(M_1)}\geq \textit{Err}_{\text{MC}(M_2)} \quad \text{ if and only if } \quad \text{prob}_{\text{MC}(M_1)} + \text{prob}_{\text{MC}(M_2)}\geq 1, 
	$$ 
	$$
	\textit{Err}_{\text{MC}(M_1)}\leq \textit{Err}_{\text{MC}(M_2)} \quad \text{ if and only if } \quad \text{prob}_{\text{MC}(M_1)} + \text{prob}_{\text{MC}(M_2)}\leq 1. 
	$$ 
\end{proof}
\section{Proofs}\label{AppC}
\begin{proof}[{\bf Proof of Proposition 4.1}]
	Here we prove the case for positive $g^\prime(\lambda t)$,  $\mu^\prime(t)$, $g^{\prime\prime}(\lambda t)$ and $\mu^{\prime\prime}(t)$, since the other case can be shown in the same way. For $\sigma(t) = \sigma$, $\lambda(t) = \lambda$ and $b(t) = b$, we can solve the ODE system (\ref{ode}), and we obtain
	\begin{align*}
	\alpha(t) = \frac{\sqrt{\lambda}}{\sigma}, \qquad \beta(t) = \beta_0 \e^{-t}\left[ 1 + \frac{1}{\sigma\sqrt{\lambda}} \int_{0}^{t} \e^{s} \mu\left( \frac{s}{\lambda} \right) \rd s \right], \qquad \gamma(t) = \frac{t}{\lambda},
	\end{align*}
	which are defined for $t \in [\lambda t_1, \lambda t_2]$.	Since $\mu(t) \in {C}^2\left([t_1, t_2]\right)$, we have
	$$g(t) = \frac{\sqrt{\lambda}}{\sigma}b - \beta(t) \in {C}^2\left([t_1, t_2]\right).$$
	Therefore,
	\begin{align*}
	\frac{\inf\limits_{t \in [\lambda t_1, \lambda t_2]}  g^{\prime\prime}(t) }{\inf\limits_{t \in [t_1, t_2]}  \mu^{\prime\prime}(t) }
	&= \frac{\inf\limits_{t \in [\lambda t_1, \lambda t_2]}  -\left(  \mu(\frac{t}{\lambda}) \frac{\beta_0}{\sigma\sqrt{\lambda}} - \beta(t)  \right)^{\prime} }{\inf\limits_{t \in [t_1, t_2]}  \mu^{\prime\prime}(t) } \\
	&= \frac{\inf\limits_{t \in [\lambda t_1, \lambda t_2]}  \mu(\frac{t}{\lambda}) \frac{\beta_0}{\sigma\sqrt{\lambda}} - \beta_0 \e^{-t}\left[ 1 + \frac{1}{\sigma\sqrt{\lambda}} \int_{0}^{t} \e^{s} \mu\left( \frac{s}{\lambda} \right) \rd s \right]  - \mu^\prime(\frac{t}{\lambda}) \frac{\beta_0}{\sigma\lambda\sqrt{\lambda}}  }{\inf\limits_{t \in [t_1, t_2]}  \mu^{\prime\prime}(t) }  \\
	&= \frac{\inf\limits_{t \in [\lambda t_1, \lambda t_2]}   \frac{\beta_0}{\sigma\sqrt{\lambda}} \left[ \mu(\frac{t}{\lambda})    -  \sigma\sqrt{\lambda} \e^{-t} - \int_{0}^{t} \e^{s-t} \mu\left( \frac{s}{\lambda} \right) \rd s   - \mu^\prime(\frac{t}{\lambda}) \frac{1}{\lambda}  \right] }{\inf\limits_{t \in [t_1, t_2]}  \mu^{\prime\prime}(t) }.
	\end{align*}
	Since $g(t)$ and $\mu(t)$ are monotone increasing and convex, if 
	$$\frac{\inf\limits_{t \in [\lambda t_1, \lambda t_2]} {\beta_0} \left[ \mu(\frac{t}{\lambda})    -  \sigma\sqrt{\lambda} \e^{-t} - \int_{0}^{t} \e^{s-t} \mu\left( \frac{s}{\lambda} \right) \rd s   - \mu^\prime(\frac{t}{\lambda}) \frac{1}{\lambda}  \right] }{\inf\limits_{t \in [t_1, t_2]}  \mu^{\prime\prime}(t) } < \sigma\sqrt{\lambda},$$
	we have
	$$\frac{\inf\limits_{t \in [\lambda t_1, \lambda t_2]}  g^{\prime\prime}(t) }{\inf\limits_{t \in [t_1, t_2]}  \mu^{\prime\prime}(t) } < 1.$$
	Hence $g'(t)<\mu'(t)$ for $t \in [t_1, t_2]$. This means that for a fixed level of accuracy, the transformed barrier $g(t)$ can be approximated by a piece-wise constant function with fewer segments than for $\mu(t)$, i.e. $g(t)$ can be approximated more efficiently with the transformation method.
\end{proof}
\begin{proof}[{\bf Proof of Proposition 4.2}]
	We first consider the case that $z,z^\prime < b$. Since the two events $\left\{ M_{0,T} \geq b \right\}$ and $\left\{ \tau_{Z,b} \leq T \right\}$ are equivalent, 
	\begin{align*}
	\mathbb{P}\left( M_{0,T} \geq b\, | \,  Z_T=z^\prime, Z_0 = z \right) = \mathbb{P}\left( \tau_{Z,b} \leq T\, | \,  Z_T=z^\prime, Z_0 = z \right)
	= \int_{0}^{T} f_{\tau_{Z,b}}\left( t\, | \,  Z_T=z^\prime, Z_0 = z \right) \rd t,
	\end{align*}
	where $f_{\tau_{Z,b}}\left( t\, | \,  Z_T=z^\prime, Z_0 = z \right)$ denotes the conditional density function of the first-passage-time $\tau_{Z,b}$. By the Bayes theorem, we have
	\begin{align*}
	\mathbb{P}\left( M_{0,T} \geq b\, | \,  Z_T=z^\prime, Z_0 = z \right)
	=& \int_{0}^{T} p\left( 0,T,z,z^\prime\, | \,  \tau_{Z,b} = t \right) \frac{
		f_{\tau_{Z,b}}(t;z)	}{p(0, T, z, z^\prime)} \rd t.
	\end{align*}
	Here, $p\left( 0,T,z,z^\prime\, | \,  \tau_{Z,b} = t \right)$ denotes the conditional transition density of $(Z_t)_{t \geq 0} $ from state $b$ at time $t$ to state $z^\prime$ at time $T$. Since $(Z_t)$ is a continuous process, we have $\left\{ \tau_{Z,b} = t \right\} \cap \left\{ Z_t \geq b \right\} = \left\{ \tau_{Z,b} = t \right\}$.
	Hence,
	$$p\left( 0,T,z,z^\prime\, | \,  \tau_{Z,b} = t \right) = p\left( 0,T,z,z^\prime\, | \,  \tau_{Z,b} = t, Z_t = b \right).$$
	Because $\left\{ \tau_{Z,b} = t \right\} \in \sigma\left( Z_s:  0 \leq s \leq t \right)$ and $\left\{ Z_T \geq z^\prime \right\} \in \sigma\left( Z_s: 
	t< s \leq T
	\right)$, by Lemma \ref{markov_ind}, we have 
	$$p\left( 0,T,z,z^\prime\, | \,  \tau_{Z,b} = t, Z_t = b \right) = p\left( 0,T,z,z^\prime\, | \,  Z_t = b \right) = p\left( t,T,b,z^\prime \right).$$
	Therefore,
	\begin{align*}
	\mathbb{P}\left( M_{0,T} \geq b\, | \,  Z_T=z^\prime, Z_0 = z \right) =& \int_{0}^{T}  \frac{p(t, T, b, z^\prime)}{p(0, T, z, z^\prime)} 
	f_{\tau_{Z,b}}(t;z) \rd t.
	\end{align*}
	In the case that either $z \geq b$ or $z^\prime \geq b$, the probability turns out to be equal to one due to the continuous property of the process $(Z_t)_{t\geq 0}$. 
\end{proof}
\begin{proof}[{\bf Proof of Corollary 5.1}]
	By Theorem \ref{main_thm}, we have\\
	\begin{align*}
	&\mathbb{P}\left( M_{t_{0},t_1} \geq b_1, \cdots, M_{t_{N-1},t_N} \geq b_N \, | \,  Z_{t_0} = z_0 \right) \\
	&= \int_{\mathbb{R}} \mathbb{P}\left( M_{t_{0},t_1} \geq b_1 \, | \,  Z_{t_{0}} = z_{0}, Z_{t_1}=z_1 \right) p(t_{0}, t_1, z_{0}, z_1) \cdots \\
	&\hspace{4cm} \times  \int_{\mathbb{R}} \mathbb{P}\left( M_{t_{N-2},t_{N-1}} \geq b_{N-1} \, | \,  Z_{t_{N-2}} = z_{N-2}, Z_{t_{N-1}}=z_{N-1} \right) \\
	&\Hquad
	\times p(t_{N-2}, t_{N-1}, z_{N-2}, z_{N-1})\left( 1-\mathbb{P}\left( M_{t_{N-1},t_N} < b_N \, | \,  Z_{t_{N-1}} = z_{N-1} \right) \right) \rd z_{N-1} \cdots \rd z_1,
	\end{align*}
	where $p(t_{i-1}, t_i,z_{i-1},z_i)$ is the transition density function of the process $(Z_t)_{t \geq 0}$. By Theorem \ref{OUhit} we have 
	\begin{align*}
	1-\mathbb{P}\left( M_{t_{N-1},t_N} < b_N \, | \,  Z_{t_{N-1}} = z_{N-1} \right)= \left[1 - \sum\limits_{k = 1}^{\infty} c_k^{(N)} \e^{-\alpha_k^{(N)} \Delta t} \mathscr{H}_{\alpha_k^{(N)}}\left( -z_{N-1} \right) \right] \mathbb{1}\left({z_{N-1} < b_N}\right) + \mathbb{1}\left({z_{N-1} \geq b_N}\right).
	\end{align*}
	By the homogeneous property of the standardized OU-process and Proposition \ref{prop_bridge}, we have 
	\begin{align*}
	&\mathbb{P}\left( M_{t_{i-1},t_{i}} \geq b_i \, | \,  Z_{t_{i-1}}=z_{i-1}, Z_{t_i} = z_i \right) p(t_{i-1}, t_i, z_{i-1}, z_i) \\
	&= \mathbb{P}\left( M_{0, \Delta t} \geq b_i \, | \,  Z_{0}=z_{i-1}, Z_{\Delta t} = z_i \right) p(0, \Delta t, z_{i-1}, z_i)\\
	&= \int_{0}^{\Delta t} p(t, \Delta t, b_i, z_i) {\mathbb{P}\left( \tau_{Z,b_i} \in \rd  t \, | \,  Z_{0} = z_{i-1} \right)} \mathbb{1}\left({z_{i-1} < b_i}\right)\mathbb{1}\left({ z_i < b_i}\right) \\
	&\quad+ {p(0, \Delta t, z_{i-1}, z_i)} \left( 1-\mathbb{1}\left({z_{i-1} < b_i}\right)\mathbb{1}\left({ z_i < b_i}\right) \right)
	\end{align*}
	and further	
	\begin{align*}
	&\mathbb{P}\left( M_{t_{i-1},t_{i}} \geq b_i \, | \,  Z_{t_{i-1}}=z_{i-1}, Z_{t_i} = z_i \right) p(t_{i-1}, t_i, z_{i-1}, z_i) \\
	&= \sum_{k=1}^{\infty}  {c_k^{(i)} \alpha_k^{(i)} \mathscr{H}_{\alpha_k^{(i)}}( -z_{N-1}) }
	\int_{\e^{-\Delta t}}^{1} \frac{x^{-\alpha_k^{(i)} - 1}\mathbb{1}\left({z_{i-1} < b_i}\right)\mathbb{1}\left({ z_i < b_i}\right)}{\sqrt{\pi(1-x^2)} }\exp\left\{ -\frac{(z_i - b_i x)^2}{1-x^2}- \alpha_k^{(i)} \Delta t \right\}  {\rd x}\\
	&\hspace{8cm}+ p(0, \Delta t, z_{i-1}, z_{i}) \left[ 1 - \mathbb{1}\left({z_{i-1} < b_i}\right)\mathbb{1}\left({ z_i < b_i}\right) \right].
	\end{align*}
	Therefore, 
	\begin{align*}
	&\mathbb{P}\left( M_{t_{0},t_1} \geq b_1, \cdots, M_{t_{N-1},t_N} \geq b_N \, | \,  Z_{t_0} = z_0 \right)\notag\\ 
	&= \int_{\mathbb{R}} \kappa\left(z_{0},  z_{1}, b_1 \right) \cdots \int_{\mathbb{R}} \kappa\left( z_{N-3},  z_{N-2}, b_{N-2} \right)\int_{\mathbb{R}} \kappa\left( z_{N-2},  z_{N-1}, b_{N-1} \right)\bar{q}\left( z_{N-1}, b_N \right) \rd z_{N-1} \rd z_{N-2} \cdots \rd z_1.
	\end{align*}
\end{proof}
\begin{proof}[{\bf Proof of Proposition 5.1}]
	We can approximate the nested integral by 
	\begin{align*}
	I \approx& \sum_{k_1 = 1}^{L} \kappa\left(z_0,  z^{(k_{1})}_1, b_{1} \right) \delta z \sum_{k_2 = 1}^{L} \kappa\left(z^{(k_{1})}_1,  z^{(k_{2})}_2, b_{2} \right) \delta z \cdots \sum_{k_{N-1} = 1}^{L} \kappa\left(z^{(k_{N-2})}_{N-2},  z^{(k_{N-1})}_{N-1}, b_{N-1} \right) \bar q\left( z^{(k_{N-1})}_{N-1}, b_N \right) \delta z.
	\end{align*}
	We write
	\begin{align*}
	f_i( z^{(k_{i})}_i) =& \sum_{k_{i+1} = 1}^{L} \kappa( z^{(k_{i})}_i,  z^{(k_{i+1})}_{i+1}, b_{{i+1}}) \sum_{k_{i+2} = 1}^{L} \kappa( z^{(k_{i+1})}_{i+1},z^{(k_{i+2})}_{i+2}, b_{{i+1}}) \cdots  \\
	&\hspace{5cm}\times\sum_{k_{N-1} = 1}^{L} \kappa( z^{(k_{N-2})}_{N-2},  z^{(k_{N-1})}_{N-1}, b_{N-1}) \bar q( z^{(k_{N-1})}_{N-1}, b_N),
	\end{align*}
	and
	$
	F_i = [
	\begin{array}{llll}
	f_i(  z_i^{(1)})& f_i(z_i^{(2)})& \cdots& f_i( z_i^{(L)})
	\end{array}
	]^\top
	$.
	Then we proceed as follows:
	\begin{itemize}
		\item[1)]
		The integral $I$ can be approximated to obtain
		$$
		I \approx \sum_{k_1 = 1}^{L} \kappa\left( z_0,  z^{(k_{1})}_1, b_{1} \right)f_1\left( z^{(k_{1})}_1 \right) = K_1 F_1.
		$$
		\item [2)]
		Then
		$f_1\left( z^{(k_{1})}_1 \right) = \sum_{k_2 = 1}^{L} \kappa\left( z^{(k_{1})}_1,  z^{(k_{2})}_2, b_{2} \right) f_2\left( z^{(k_{2})}_2 \right),$
		and therefore
		$
		F_1= K_2 F_2.
		$
		\item[3)]
		We repeat step 2) for $N-2$ times until 
		\begin{align*}
		f_{N-3}\left( z^{(k_{N-3})}_{N-3} \right) = \sum_{k_{N-2} = 1}^{L} \kappa\left( z^{(k_{N-3})}_{N-3},  z^{(k_{N-2})}_{N-2}, b_{N-2} \right) f_{N-2}\left( z^{(k_{N-2})}_{3} \right) = K_{N-2} F_{N-2},
		\end{align*}
		and thus, $F_{N-2} = K_{N-1}\bar Q.$
	\end{itemize}
	We finally have $I \approx \prod_{i=1}^{N-1}K_i  \bar Q (\delta z)^{N-1}$. Here, we may approximate $q\left( z_{N-1}, b_N \right)$ and $\kappa\left( z_{i-1},  z_{i}, b_i \right)$ as follows: 
	\begin{align*}
	\bar{q}\left( z_{N-1}, b_N \right) \approx&
	\left[1 - \sum\limits_{k = 1}^{K} c_k^{(N)} \e^{-\alpha_k^{(N)} t} \mathscr{H}_{\alpha_k^{(N)}}\left( -z_{N-1} \right) \right] \mathbb{1}\left({z_{N-1} < b_N}\right) + \mathbb{1}\left({z_{N-1} \geq b_N}\right), \\
	\kappa(z_{i-1}, z_i, b_i) \approx& \sum_{k=1}^{K}  {c_k^{(i)} \alpha_k^{(i)} \mathscr{H}_{\alpha_k^{(i)}}( -z_{N-1}) }\\
	&\times
	\sum_{j=1}^{J}
	\frac{x_j^{-\alpha_k^{(i)} - 1}\mathbb{1}\left({z_{i-1} < b_i}\right)\mathbb{1}\left({ z_i < b_i}\right)}{\sqrt{\pi(1-x_j^2)} }\exp\left\{ -\frac{(z_i - b_i x_j)^2}{1-x_j^2}- \alpha_k^{(i)} \Delta t \right\}  {\delta x_j} \\
	&+ p(0, \Delta t, z_{i-1}, z_{i}) \left(1 - \mathbb{1}\left({z_{i-1} < b_i}\right)\mathbb{1}\left({ z_i < b_i}\right) \right).
	\end{align*}
	where $\e^{-\Delta t}=x_0 < x_1 < \cdots < x_J = 1$ and $\delta x_j = x_j - x_{j-1}$. 
\end{proof}

\begin{thebibliography}{9}
	
\bibitem[\protect\citeauthoryear{Abramowitz and Stegun}{1964}]{Abramowitz} Abramowitz, M. and Stegun, I. A., 1964. Handbook of mathematical functions: with formulas, graphs, and mathematical tables (Vol. 55). Washington, D. C.: Courier Corporation.

\bibitem[\protect\citeauthoryear{Alili et al.}{2005}]{Alili} Alili, L., Patie, P. and Pedersen, J. L., 2005. Representations of the first hitting time density of an Ornstein-Uhlenbeck process. Stochastic Models, 21(4), pp. 967-980.

\bibitem[\protect\citeauthoryear{Bingham and Kiesel}{2013}]{bingham} Bingham, N. H. and Kiesel, R., 2013. Risk-neutral valuation: Pricing and hedging of financial derivatives. London: Springer Science \& Business Media.

\bibitem[\protect\citeauthoryear{Buonocore et al.}{1987}]{Buonocore} Buonocore, A., Nobile, A. G. and Ricciardi, L. M., 1987. A new integral equation for the evaluation of first-passage-time probability densities. Advances in Applied Probability, 19(4), pp. 784-800.

\bibitem[\protect\citeauthoryear{Di Nardo et al.}{2001}]{DiNardo} Di Nardo, E., Nobile, A. G., Pirozzi, E. and Ricciardi, L. M., 2001. A computational approach to first-passage-time problems for Gauss-Markov processes. Advances in Applied Probability, 33, pp. 453-482.

\bibitem[\protect\citeauthoryear{Durbin}{1985}]{Durbin}
Durbin, J., 1985. The first-passage density of a continuous Gaussian process to a general boundary. Journal of Applied Probability, 22, pp. 99-122.
%
\bibitem[\protect\citeauthoryear{Durbin and Williams}{1992}]{DurbinWilliams}
Durbin, J. and Williams, D., 1992. The first-passage density of the Brownian motion process to a curved boundary. Journal of Applied Probability, 29, pp. 291-304.

\bibitem[\protect\citeauthoryear{Giorno et al.}{1989}]{Giorno} Giorno, V., Nobile, A. G., Ricciardi, L. M. and Sato, S., 1989. On the evaluation of first-passage-time probability densities via non-singular integral equations. Advances in Applied Probability, 21, pp. 20-36.

\bibitem[\protect\citeauthoryear{G{\"o}ing-Jaeschke and Yor}{2003}]{GoingJaeschke} G{\"o}ing-Jaeschke, A. and Yor, M., 2003. A clarification note about hitting times densities for Ornstein-Uhlenbeck processes. Finance and Stochastics, 7(3), pp. 413-415.

\bibitem[\protect\citeauthoryear{Guti{\'e}rrez et al.}{1997}]{Gutierrez} Guti{\'e}rrez, R., Ricciardi, L. M., Rom{\'a}n, P. and Torres, F., 1997. First-passage-time densities for time-non-homogeneous diffusion processes. Journal of Applied Probability, 34(3), pp. 623-631.

\bibitem[\protect\citeauthoryear{Hernandez-del-Valle}{2012}]{HernandezdelValle} Hernandez-del-Valle, G., 2012. On the first time that an Ito process hits a barrier. arXiv preprint arXiv:1209.2411.

\bibitem[\protect\citeauthoryear{Karlin and Taylor}{1981}]{2ndCourse} Karlin, S. and Taylor, H.E., 1981. A second course in stochastic processes. New York; London: Academic Press.

\bibitem[\protect\citeauthoryear{Kent}{1980}]{Kent} Kent, J. T., 1980. Eigenvalue expansion for diffusion hitting times. Zeitschrift der Wahrscheinlichkeitstheorie und verwandte Gebiete, 52, pp. 309-319.

\bibitem[\protect\citeauthoryear{Leblanc et al.}{2000}]{Leblanc} Leblanc, B., Renault, O. and Scaillet, O., 2000. A correction note on the first passage time of an Ornstein-Uhlenbeck process to a boundary. Finance and Stochastics, 4(1), pp. 109-111.

\bibitem[\protect\citeauthoryear{Lebedev and Silverman}{1972}]{Lebedev} Lebedev, N. N. and Silverman, R. A., 1972. Special functions and their applications. London: Courier Corporation.

\bibitem[\protect\citeauthoryear{Lehmann}{2002}]{Lehmann} Lehmann, A., 2002. Smoothness of first passage time distributions and a new integral equation for the first passage time density of continuous Markov processes. Advances in Applied Probability, 34, pp. 869-887.

\bibitem[\protect\citeauthoryear{Lindel\"of}{1894}]{Lindelof} Lindel\"of, E., 1894. Sur l'application de la m\'ethode des approximations successives aux \'equations diff\'erentielles ordinaires du premier ordre. Comptes rendus hebdomadaires des s\'eances de l'Acad\'emie des Sciences, 116(3), pp. 454-457.

\bibitem[\protect\citeauthoryear{Linetsky}{2004a}]{Linetsky2004b} Linetsky, V., 2004a. Lookback options and diffusion hitting times: A spectral expansion approach. Finance and Stochastics, 8(3), pp. 373-398.

\bibitem[\protect\citeauthoryear{Linetsky}{2004b}]{Linetsky2004d} Linetsky, V., 2004b. Computing hitting time densities for CIR and OU diffusions: applications to mean-reverting models. Journal of Computational Finance, 7, pp. 1-22.

\bibitem[\protect\citeauthoryear{Lipton and Kaushansky}{2018}]{LipKau} Lipton, A. and Kaushansky, V., 2018. On the first hitting time density of an Ornstein-Uhlenbeck Process. arXiv:1810.02390v2.

\bibitem[\protect\citeauthoryear{Lo and Hui}{2006}]{lo} Lo, C. F. and Hui, C. H., 2006. Computing the first passage time density of a time-dependent Ornstein-Uhlenbeck process to a moving boundary. Applied Mathematics Letters, 19(12), pp. 1399-1405.

\bibitem[\protect\citeauthoryear{Martin et al.}{2018}]{Martin} Martin, R. J., Kearney, M. J. and Craster, R. V., 2018. Long- and short-time asymptotics of the first-passage time of the Ornstein-Uhlenbeck and other mean-reverting processes. Journal of Physics A: Mathematical and Theoretical, 52(134001).

\bibitem[\protect\citeauthoryear{{\O}ksendal}{2003}]{oksendal} {\O}ksendal, B., 2003. Stochastic differential equations. In Stochastic differential equations (pp. 65-84). Springer, Berlin, Heidelberg.

\bibitem[\protect\citeauthoryear{Patie}{2004}]{Patie} Patie, P., 2004. On some first passage time problems motivated by financial applications (Doctoral dissertation, Universit{\"a}t Z{\"u}rich).

\bibitem[\protect\citeauthoryear{Peters and Shevchenko}{2015}]{Garethbook} Peters, G. W. and Shevchenko, P. V., 2015. Advances in heavy tailed risk modeling: a handbook of operational risk. Hoboken, New Jersey: John Wiley \& Sons.

\bibitem[\protect\citeauthoryear{Pitman and Yor}{1981}]{PitmanYor1981} Pitman, J. and Yor, M., 1981. Bessel processes and infinitely divisible laws. In Stochastic integrals (pp. 285-370). Springer, Berlin, Heidelberg.

\bibitem[\protect\citeauthoryear{Pitman and Yor}{1982}]{PitmanYor1982} Pitman, J. and Yor, M., 1982. A decomposition of Bessel bridges. Zeitschrift f{\"u}r Wahrscheinlichkeitstheorie und verwandte Gebiete, 59(4), pp. 425-457.

\bibitem[\protect\citeauthoryear{Ricciardi and Sato}{1988}]{Ricciardi} Ricciardi, L. M. and Sato, S., 1988. First-passage-time density and moments of the Ornstein-Uhlenbeck process. Journal of Applied Probability, 25(1), pp. 43-57.

\bibitem[\protect\citeauthoryear{Tuckwell and Wan}{1984}]{Tuckwell} Tuckwell, H. C. and Wan, F. Y., 1984. First-passage time of Markov processes to moving barriers. Journal of Applied Probability, 21(4), pp. 695-709.

\bibitem[\protect\citeauthoryear{Wenocur}{1987}]{Eigen} Wenocur, M. L., 1987. Diffusion first passage times: approximations and related differential equations. Stochastic processes and their applications, 27, pp. 159-177.

\bibitem[\protect\citeauthoryear{Yi}{2010}]{yi} Yi, C., 2010. On the first passage time distribution of an Ornstein-Uhlenbeck process. Quantitative Finance, 10(9), pp. 957-960.

\bibitem[\protect\citeauthoryear{Zaitsev and Polyanin}{2002}]{ODEs} Zaitsev, V. F. and Polyanin, A. D., 2002. Handbook of exact solutions for ordinary differential equations. Boca Raton: CRC press.
\end{thebibliography}
\end{document}